\documentclass{colt2015} % Anonymized submission

\usepackage{algorithm,algorithmic}
\usepackage{amssymb, multirow, paralist, color}

\newtheorem{prop}{Proposition}
\newtheorem{lem}{Lemma}
\newtheorem{cor}{Corollary}
\newtheorem{ass}{Assumption}
\usepackage{enumitem}
\usepackage{diagbox}
\usepackage{booktabs}
\usepackage{comment}

\def \R {\mathbb{R}}

\def \v {\mathbf{v}}

\def \x {\mathbf{x}}

\def \E {\mathrm{E}}

\def \x {\mathbf{x}}

\def \b {\mathbf{b}}

\def \d {\mathbf{d}}
\def \z {\mathbf{z}}
\def \s {\mathbf{s}}

\def \y {\mathbf{y}}

\def \g {\mathbf{g}}

\def \xh {\widehat{\x}}

\def \U {\mathbf U}

\def \y {\mathbf{y}}
\def \E {\mathrm{E}}
\def \x {\mathbf{x}}
\def \g {\mathbf{g}}

\def \z {\mathbf{z}}

\def \R {\mathbb{R}}

\def \v {\mathbf{v}}

\def \d {\mathbf{d}}

\def \b {\mathbf{b}}

\def \s {\mathbf{s}}

\def \xh {\widehat{\x}}

\begin{document}

\title[Noisy Negative Curvature Descent Competing with Gradient Descent]{On Noisy Negative Curvature Descent: Competing with Gradient Descent for Faster Non-convex Optimization}
 \author{\Name{Mingrui Liu}\Email{mingrui-liu@uiowa.edu}\\
% \Name{Qihang Lin}$^\ddagger$ \Email{qihang-lin@uiowa.edu}\\
\Name{Tianbao Yang}\Email{tianbao-yang@uiowa.edu}\\
   \addr Department of Computer Science, The University of Iowa, Iowa City, IA 52242 \\
  % \addr$^\ddagger$Department of Management  Sciences \\
 %   The University of Iowa, Iowa City, IA 52242
}

%\title{Improved High Probability Iteration Complexity for\\ Stochastic Subgradient Method}
%\author{Tianbao  Yang\\ Department of Computer Science\\The University of Iowa, Iowa City, IA 52242 \\tianbao-yang@uiowa.edu}
%\date{May 19, 2016}
\maketitle
\vspace*{-0.5in}
\begin{center}{First version: September 25, 2017}\end{center}
%\vspace*{-0.2in}
%\begin{center}{Second version: February 3, 2017}\end{center}

\begin{abstract}
%There has been a growing interest on establishing second-order convergence results for non-convex optimization with reduced time complexity. Several recent works focus on leveraging the Hessian-vector product to find a second-order stationary solution with strong complexity guarantee (e.g., almost linear time complexity in the problem's dimensionality), which avoids computing the Hessian matrix in conventional algorithms (e.g.,  cubic regularization). 
The Hessian-vector product has been utilized to find a second-order stationary solution with strong complexity guarantee (e.g., almost linear time complexity in the problem's dimensionality). In this paper, we propose to further reduce the number of Hessian-vector products for faster non-convex optimization. %In particular, we employ the Hessian-vector product to compute the negative curvature direction (i.e., the eigen-vector corresponding to the smallest eigen-value of the Hessian) for decreasing the objective value similar in spirit to many previous papers. 
Previous works need to approximate the smallest eigen-value with a sufficient precision (e.g., $\epsilon_2\ll 1$) in order to achieve a sufficiently accurate second-order stationary solution (i.e., $\lambda_{\min}(\nabla^2 f(\x))\geq -\epsilon_2)$. In contrast, the proposed algorithms only need to compute the smallest eigen-vector approximating the corresponding eigen-value up to a small power of current gradient's norm. As a result, it can dramatically reduce the number of Hessian-vector products during the course of optimization before reaching first-order stationary points (e.g., saddle points). The key building block of the proposed algorithms is a novel updating step named the NCG step, which  lets a noisy negative curvature descent compete with the gradient descent. %By leveraging the basic NCG step, we present two algorithms each with two variants  that build up incrementally with improving  time complexities. 
We show that  in the worst-case  the proposed algorithms with their favorable prescribed accuracy requirements can match the best time complexity in literature for achieving a second-order stationary point but with an arguably smaller per-iteration cost. We also show that the proposed algorithms can benefit from inexact Hessian (e.g., the sub-sampled Hessian for a finite-sum problem) by developing their variants accepting inexact Hessian under a mild condition for achieving the same goal.  Moreover, we develop a stochastic algorithm for a finite or infinite sum non-convex optimization problem, which only involves computing a sub-sampled gradient and a noisy negative curvature of a sub-sampled Hessian at each iteration. To the best of our knowledge, the proposed stochastic algorithm is the first one, which converges to a second-order stationary point in {\it high probability} with a time complexity independent of the sample size and almost linear in the problem's dimensionality.

%In this draft, we propose a novel update in non-convex optimization where a gradient descent step competes with a {\bf noisy} negative curvature descent step to determine the descent direction, to which we refer as GNC step. Based on the proposed GN-competing step, we develop several algorithms with arguably faster convergence of non-convex optimization to the first-order and second-order stationary solutions than the state-of-the-art  ``dimension"-free non-convex optimization algorithms. A remarkable result of one variant is that  it allows arbitrary accuracy level of both the first-order optimality and the second-order optimality, which could lead to faster convergence for some strict-saddle functions (e.g., matrix factorization). 
%variant of non-convex optimization algorithm, which at each iteration greedily takes a gradient descent step or a negative curvature descent step depending on which gives faster decrease in the objective value. Then we apply the algorithm for solving one important family of non-convex optimization problems, namely strict-saddle, and establish improved convergence to the second-order stationary solutions while maintaining the efficiency of the gradient descent method. The benefit of the proposed algorithm is that it allows arbitrary accuracy level of both the first-order optimality and the second-order optimality. 
\end{abstract}

\section{Introduction}
In this paper, we consider the following optimization problem: 
\begin{align}\label{eqn:opt}
\min_{\x\in\R^d}f(\x),
\end{align}
where $f(\x)$ is a non-convex smooth function with Lipschitz continuous Hessian. A standard measure of an optimization algorithm is how fast the algorithm converges to an optimal solution. However, finding the global optimal solution to a generally non-convex problem is intractable~\citep{nemirovskiĭ1983problem} and is even a NP-hard problem~\citep{Hillar:2013:MTP:2555516.2512329}. Therefore, similar to many previous works, we aim to find an approximate second-order stationary point with: 
\begin{align}\label{eqn:sog}
    \|\nabla f(\x)\|\leq \epsilon_1, \quad \text{ and }\quad \lambda_{\min}(\nabla^2 f(\x))\geq -\epsilon_2,
\end{align}
which nearly satisfy  the second-order necessary conditions for optimality, i.e., $\nabla f(\x_*)=0, \lambda_{\min}(\nabla^2 f(\x_*))\geq 0$, where $\|\cdot\|$ denotes the Euclidean norm and $\lambda_{\min}(\cdot)$ denotes the smallest eigen-value function. In the sequel, we refer to a solution that satisfies~\eqref{eqn:sog} as an $(\epsilon_1, \epsilon_2)$-second-order stationary solution. When the function is non-degenerate (i.e., strict saddle or the Hessian at all saddle points have a strictly negative eigen-value), then the solution satisfying~(\ref{eqn:sog}) is close to a local minimum for sufficiently small $0<\epsilon_1, \epsilon_2\ll 1$. Please note that in this paper we do not follow the tradition of~\citep{nesterov2006cubic} that restricts $\epsilon_2 = \sqrt{\epsilon_1}$. One reason is for more generality that allows us to compare several recent results and another reason is that having different accuracy levels for the first-order and the second-order guarantee brings more flexibility in the choice of our algorithms as demonstrated later.  

Recently, there has emerged a surge of studies interested in finding  an approximate second-order stationary point that satisfy~\eqref{eqn:sog}. To motivate the present work, below we briefly review several works that enjoy strong complexity guarantee by leveraging the second-order information carried by the Hessian. More related results and comparisons with these results are deferred to latter sections. The best known {\bf iteration complexity} for the second-order guarantee dates back to the cubic regularization method~\citep{nesterov2006cubic}, which requires solving an $O(1/\epsilon^{3/2})$ number of cubic regularization steps for finding an $(\epsilon, \sqrt{\epsilon})$-second-order stationary solution. This result improves earlier work on trust-region methods that involve solving a trust-region problem at each iteration~\citep{conn2000trust,Cartis:2012:CBS:2076038.2076303}. However, solving the cubic regularization step costs a time  that is super-linear of the dimensionality of the problem. How to reduce the time complexity of the cubic regularization approach has been a central theme of several recent works~\citep{DBLP:conf/stoc/AgarwalZBHM17,DBLP:journals/corr/CarmonDHS16,peng16inexacthessian} that are discussed in order.   
\cite{DBLP:conf/stoc/AgarwalZBHM17} developed a fast algorithm for approximately solving the cubic regularization step that only involves Hessian-vector product, which can be carried out in linear time of the dimensionality for many interesting machine learning problems (e.g., deep learning). \cite{DBLP:journals/corr/CarmonDHS16} took a different route by leveraging the negative curvature descent for reaching a region where the function is almost convex and accelerated gradient method for an almost-convex function. The negative curvature descent employs the eigen-vector corresponding to the smallest eigen-value of the Hessian, which also only involves the Hessian-vector product. Both of these works enjoy a {\bf time complexity} of $\widetilde O(T_h/\epsilon^{7/4})$~\footnote{where $\widetilde O$ suppresses logarithmic term on $d$ and $1/\epsilon$. } for finding an $(\epsilon, \sqrt{\epsilon})$-second-order stationary solution, where $T_h$ denotes the runtime of a Hessian-vector product. Another direction to reduce the time complexity of cubic regularization step for a finite-sum problem in machine learning is to employ a sub-sampled Hessian. \cite{peng16inexacthessian} showed that by employing an inexact Hessian with a sufficient approximation accuracy, the cubic regularization method and conventional trust-region method can enjoy the same iteration complexity of their counterparts with an exact Hessian. In addition, \cite{peng16inexacthessian} also used the negative gradient direction and the negative curvature direction to approximately solve the involved sub-problems at each iteration. By investigating these three works deeply, one can find that they all need to compute the negative curvature (i.e., the eigen-vector corresponding to the smallest eigen-value of the Hessian) to approximate the smallest eigen-value with sufficient accuracy. In particular, in order to achieve an $(\epsilon_1,\epsilon_2)$-second-order stationary point, their algorithms require an $\widetilde O(1/\sqrt{\epsilon_2})$ number of Hessian-vector products with favorable algorithms (e.g., the Lanczos algorithm) for finding the negative curvature or solving the involved sub-problems approximately. 

\begin{table*}[t]
		\caption{Comparison with existing work leveraging the second-order information for achieving an $(\epsilon_1, \epsilon_2)$-second-order stationary point, where  $T_h$ denotes the runtime for computing a Hessian-vector product, and  %$T_{\text{cubic}}$ refers to the time complexity for solving a cubic regularization step up to a certain accuracy, which is unspecified in the original work, 
		$\alpha\in(0, 1]$, $\omega$ is the matrix multiplication constant. na means not available in the original paper, $\x_t$ denotes the current iterate at which the negative curvature is required. Results marked by $^*$ only show the order of iterations that involve the Hessian-vector products and the corresponding per-iteration time complexity shows the time complexity for each such iteration.} %We simply put them in one line. }
		\centering
		\label{tab:data}
		\begin{scriptsize}\begin{tabular}{l|lll}
			\toprule
			%&This paper&\citep{besbes-2013-optimization}&\citep{DBLP:conf/aistats/JadbabaieRSS15}\\
			% \midrule
			%\multicolumn{2}{r|}{Definition of variation} &path variation &functional variation \\
			%\bottomrule
			%\toprule
			algo.&inexact Hessian &per-iteration &iteration complexity for\\
			&& time complexity&having $(\epsilon_1,\epsilon_2)$ convergence\\
			\midrule
			\cite{nesterov2006cubic} &na&$ O\left(d^\omega\right)$&$O\left(\frac{1}{\epsilon_1^{3/2}}\right)$ ($\epsilon_2=\sqrt{\epsilon_1}$)\\ 
			\midrule
			%Strongly Convex& Stochastic Gradient &N.A.&$O((V^f_T)^{1/2}T^{1/2})$\\
			\cite{DBLP:conf/stoc/AgarwalZBHM17} &na&$\widetilde O\left(\frac{T_h}{\sqrt{\epsilon_2}}\right)$&$O\left(\frac{1}{\epsilon_1^{3/2}}\right)$ ($\epsilon_2=\sqrt{\epsilon_1}$)\\ 
			\midrule
			\cite{DBLP:journals/corr/CarmonDHS16} &na&$\widetilde O\left(\frac{T_h}{\sqrt{\epsilon_2}}\right)$&$O\left(\frac{1}{\epsilon_1\epsilon_2}+\frac{1}{\epsilon_2^3}\right)^*$\\
			\midrule
			\cite{peng16inexacthessian}&yes&$\widetilde O\left(\frac{T_h}{\sqrt{\epsilon_2}}\right)$&$ O\left(\frac{1}{\epsilon_{1}^2\epsilon_2}+\frac{1}{\epsilon_2^{3}}\right)$\\
			(Alg. 1 \& Th. 1)&&\\
			\midrule
			\cite{peng16inexacthessian}&yes&$\widetilde O\left(\frac{T_h}{\sqrt{\epsilon_2}}\right)$&$ O\left(\frac{1}{\epsilon_{1}^2}+\frac{1}{\epsilon_2^{3}}\right)$\\
			(Alg. 2 \& Th. 2)&&\\
			\midrule
			%\cite{peng16inexacthessian}&yes&$T_{\text{cubic}}$&$\widetilde O\left(\left(\frac{1}{\epsilon_{1}^{3/2}}+\frac{1}{\epsilon_2^{3}}\right)T_{\text{cubic}}\right)$\\
			%(Alg. 2 \& Th. 3)&&\\
			%\midrule			
			NCG-A1& yes&$\widetilde O\left(\frac{T_h}{\sqrt{\max(\epsilon_2, \|\nabla f(\x_t)\|)}}\right)$&$ O\left(\frac{1}{\epsilon_{1}^2}+\frac{1}{\epsilon_2^{3}}\right)$\\
			\midrule
			NCG-A2& yes&$\widetilde O\left(\frac{T_h}{\sqrt{\max(\epsilon_2, \|\nabla f(\x_t)\|^\alpha)}}\right)$&$ O\left(\frac{1}{\epsilon_{1}^2}+\frac{1}{\epsilon_2^{3}}\right)$  ($\epsilon_2 = \epsilon_1^\alpha$)\\
			%NCG-B& yes&$\widetilde O\left(\frac{1}{\sqrt{\max(\epsilon_2, \|\nabala f(\x_t)\|^\alpha)}}\right)$&$\widetilde O\left(\frac{1}{\epsilon_{1}^2}T_g+\frac{1}{\epsilon_2^{7/2}}T_h\right)$ ($\epsilon_2 = \epsilon_1^\alpha$)\\
			\midrule
			NCG-B1& yes&$\widetilde O\left(\frac{T_h}{\sqrt{\max(\epsilon_2, \|\nabla f(\x_t)\|)}}\right)$&$ O\left(\frac{1}{\epsilon_{1}\epsilon_2}+\frac{1}{\epsilon_2^{3}}\right)^*$\\
			\midrule
			NCG-B2& yes&$\widetilde O\left(\frac{T_h}{\sqrt{\max(\epsilon_2, \|\nabla f(\x_t)\|^{2/3})}}\right)$&$ O\left(\frac{1}{\epsilon_{1}\epsilon_2}+\frac{1}{\epsilon_2^{3}}\right)^*$  ($\epsilon_2 = \epsilon_1^\alpha$)\\
			\bottomrule
		\end{tabular}
		\end{scriptsize}
		%\vspace*{-0.2in}
	\end{table*}
The goal of this paper is aligned with these works for reducing the time complexity of methods leveraging the second-order information, in particular the negative curvature direction. However, we aim to further reduce the number of Hessian-vector products for each iteration that requires the second-order information. The novel result in this paper is that for each such iteration the proposed algorithms only need an $\widetilde O(1/\sqrt{\max(\epsilon_2, \|\nabla f(\x_t)\|^\alpha)}$ number of Hessian-vector products, where $\alpha$ is some constant between (0, 1] and $\x_t$ is the current iterate at which the negative curvature is computed. %is either $1$ or a prescribed parameter within $(0,1]$ that characterizes $\epsilon_2 = {\epsilon_1}^\alpha$, and $\x_t$ is the current solution. 
In other words, they only need to approximate the smallest eigen-value with an additive noise up to an order of $\max(\epsilon_2, \|\nabla f(\x_t)\|^\alpha)$. As a result, before reaching the neighborhood of first-order stationary points (i.e., $\|\nabla f(\x_t)\|^\alpha>\epsilon_2$), the number of Hessian-vector products can be dramatically less than $\widetilde O(1/\sqrt{\epsilon_2})$ for iterations requiring the second-order information.  The key building block of the proposed algorithms is a novel step (named the NCG step) that greedily takes either a noisy negative curvature direction or the negative gradient direction in which that decrease the objective value most. By specifying the noise level in calculating the negative curvature direction differently and combining existing techniques (in particular, alternating the NCG steps and  the accelerated gradient descent steps), we design two algorithms (each with several variants) with improving  time complexity in their order. But they all achieve the best time complexity in their own favorable region of the input accuracy levels for the first- and second-order guarantee. Furthermore, we show that the proposed algorithms can also benefit from inexact and sub-sampled  Hessian for further reducing the time complexity. The requirement on inexact  Hessian matches that in~\citep{peng16inexacthessian}. However, our algorithms need a much less number of Hessian-vector products than that required in~\citep{peng16inexacthessian}. A quick comparison of the proposed algorithms and their properties with the four aforementioned works~\citep{nesterov2006cubic,DBLP:conf/stoc/AgarwalZBHM17,DBLP:journals/corr/CarmonDHS16,peng16inexacthessian} are shown in Table~\ref{tab:data}. Moreover, we also develop a stochastic algorithm for a stochastic non-convex optimization problem, which only involves computing a sub-sampled gradient and a noisy negative curvature of a sub-sampled Hessian at each iteration.  To the best of our knowledge, the proposed stochastic algorithm is the first one, which converges to a second-order stationary point in {\bf high probability} with a time complexity independent of the sample size and almost linear in the problem's dimensionality.  A quick comparison of the proposed stochastic algorithm with existing competitive stochastic algorithms is shown in Table~\ref{tab:2}. Detailed and more comparisons with other related work are postponed into later sections. 
\begin{table*}[t]
		\caption{Comparison with existing stochastic algorithms for achieving an $(\epsilon_1, \epsilon_2)$-second-order stationary solution to $\min_{\x}\E_i[f_i(\x)]$, where $p$ is a number at least $4$, SG denotes the stochastic gradient oracle and SHVP denotes the stochastic Hessian-vector product oracle,  $T^1_h$ denotes the runtime for computing a stochastic Hessian-vector product $\nabla^2 f_i(\x)\v$ and $T^1_g$ denotes the runtime for computing a stochastic gradient $\nabla f_i(\x)$.} %We simply put them in one line. }
		\centering
		\label{tab:2}
		\begin{scriptsize}\begin{tabular}{l|lll}
			\toprule
			%&This paper&\citep{besbes-2013-optimization}&\citep{DBLP:conf/aistats/JadbabaieRSS15}\\
			% \midrule
			%\multicolumn{2}{r|}{Definition of variation} &path variation &functional variation \\
			%\bottomrule
			%\toprule
			algo.& oracle & second-order guarantee in &time complexity\\
			&& expectation or high probability&\\
			\midrule
			Noisy SGD~\citep{pmlr-v40-Ge15} &SG&$(\epsilon, \epsilon^{1/4})$, expectation&$\widetilde O\left(T^1_gd^p\epsilon^{-4}\right)$\\ 
			\midrule
			Natasha2~\citep{natasha2} &SG + SHVP&$(\epsilon, \epsilon^{1/2})$, expectation&$\widetilde O\left( T^1_g\epsilon^{-3.5}+T^1_h\epsilon^{-2.5} \right)$\\ 
			 \midrule
            SNCG&SG + SHVP&$(\epsilon, \epsilon^{1/2})$, high probability&$\widetilde O\left(T^1_g\epsilon^{-4} + T^1_h\epsilon^{-3}\right)$\\
			\bottomrule
		\end{tabular}
		\end{scriptsize}
		%\vspace*{-0.2in}
	\end{table*}

The remainder of the paper is organized as follows. In Section~\ref{sec:rw} we review more related works. In Section~\ref{sec:pre}, we present some preliminaries and existing results on the gradient descent and the negative curvature descent methods for non-convex optimization. Section~\ref{sec:NCG} presenst the basic NCG step and  the first variant of the first algorithm NCG-A. In Section~\ref{sec:NCGb} and~\ref{sec:NCGc}, we present algorithms with improving time complexity as the order of $\epsilon_2$ increases. Section~\ref{sec:ext} presents extensions and analysis of algorithms using an inexact Hessian, and using stochastic gradient and stochastic Hessian, and also presents more comparisons with existing results. We %present some experimental results in Section~\ref{sec:exp} and 
conclude in Section~\ref{sec:conc}.

\section{More Related Work}\label{sec:rw}
In this section, we review more related work on non-convex optimization with guarantee of second-order convergence results. These works can be categorized into first-order based methods that only utilize the first-order information (e.g., the gradient) and second-order based methods that utilizes both the first- and second-order information. 

Earlier work using second-order information for non-convex optimization revolves around trust-region methods~\citep{conn2000trust}, in which each iteration solves a non-convex constrained quadratic problem known as the trust-region problem. The classical iteration complexity of trust-region methods for finding an $(\epsilon_1, \epsilon_2)$-second-order stationary point is $O(\max(\epsilon_1^{-2}\epsilon_2^{-1}, \epsilon_{2}^{-3}))$. Solving the trust-region problem may take a super-liner time in the dimensionality of the problem. For example, \cite{Hazan2016} designed a first-order method that finds an $\varepsilon$-approximate solution to the trust-region problem with $O(N/\sqrt{\varepsilon})$ time complexity, where $N$ is the number of non-zero elements in the involved Hessian matrix that could be $d^2$. Recently, \cite{peng16inexacthessian} designed an efficient variant of the trust-region method that allows for an inexact Hessian with the same iteration complexity  $O(\max(\epsilon_1^{-2}\epsilon_2^{-1}, \epsilon_{2}^{-3}))$. They also showed that the trust-region problem can be solved approximately by searching a direction in a two dimentional space consisting of the negative gradient and the negative curvature. Recent proposals on different variants of trust-region methods~\citep{Curtis2017,RePEc:spr:jglopt:v:68:y:2017:i:2:d:10.1007_s10898-016-0475-8}  match the iteration complexity of~\citep{nesterov2006cubic}'s cubic regularization method, i.e., $O(\max(\epsilon_1^{-3/2}, \epsilon_2^{-3}))$. Nonetheless, their per-iteration costs could be still super-linear of the problem's dimensionality.   

Another line of second-order methods are developed based on the cubic regularization method~\citep{nesterov2006cubic}. \cite{Cartis2011,Cartis2011b} proposed adaptive cubic regularization methods that use an adaptive estimation of the local Lipschitz constant of the Hessian and an approximate global minimizer of a local cubic regularization step. In order to guarantee the same worst-case iteration complexity as the original cubic regularization method for finding an approximate second-order stationary point, they have to solve the cubic regularized sub-problem to satisfy a certain criterion. Although the sub-problem can be solved over a successively expanded subspace generated by the Lanczos method, it is unclear how many Hessian-vector products are sufficient to satisfy their designed criterion.    \cite{DBLP:conf/icml/KohlerL17,peng16inexacthessian} have extended the adaptive cubic regularization method to variants using inexact Hessian or sub-sampled Hessian with a sufficient approximation accuracy. The variant in~\citep{peng16inexacthessian} (their Algorithm 2)  improves that in~\citep{DBLP:conf/icml/KohlerL17} in several aspects. When the sub-problem is solved approximately by searching over a two-dimensional space spanned by the negative gradient and the negative curvature directions, they obtain an iteration complexity of $O(\epsilon_1^{-2}, \epsilon_2^{-3})$ that is shown in Table~\ref{tab:data}. They also showed by solving sub-problem up to a certain accuracy similar to that in~\citep{Cartis2011,Cartis2011b}, an optimal iteration complexity of $O(\max(\epsilon_1^{-3/2}, \epsilon_2^{-3}))$ can be achieved leaving the time complexity for each iteration unclear. 
\cite{DBLP:journals/corr/CarmonD16} analyzed a gradient descent method for solving the cubic regularization step of~\citep{nesterov2006cubic}. When employing this algorithm in cubic regularization method, the time complexity of their algorithm for finding an $(\epsilon, \sqrt{\epsilon})$-second-order stationary points is $O(T_h/\epsilon^2)$, which is worse than the result in~\citep{DBLP:conf/stoc/AgarwalZBHM17}.

The use of negative curvature for  non-convex optimization to escape from non-degenerate saddle points is not new. There have been several recent studies that  explicitly explore the negative curvature direction for updating the solution. Here, we emphasize the differences between the development in this paper and previous works.  \cite{curtis2017negative} proposed a similar algorithm to our NCG-A algorithm except for how to compute the negative curvature, which dynamically selects the negative gradient or the negative curvature direction for updating the solution. The key differences between our work and \citep{curtis2017negative} lie at (i) they did not provide explicit time complexity for finding an approximate second-order stationary point, but rather give asymptotic convergence to the second-order stationary points; (ii) they ignored the computational costs for computing the (approximate) negative curvature. By assuming that the negative curvature is computed exactly, a simple analysis of their {\it iteration complexity} shows that it is similar to that of our NCG-A algorithm. In addition, they also considered stochastic version of their algorithms but provided no second-order convergence guarantee. In contrast, we also develop a stochastic algorithm with provable second-order convergence guarantee. \cite{clement17} proposed an algorithm that utilizes the negative gradient direction, the negative curvature direction, the Newton direction and the regularized Newton direction together with line search in a unified framework, and also analyzed the time complexity of a variant with inexact calculations of the negative curvature by the Lanczos algorithm and the (regularized) Newton directions by conjugate gradient method. The comparison between their algorithm and our algorithms shows that (i) we only use the gradient and the negative curvature directions; (ii) the time complexity for computing an approximate negative curvature in their work is also of the order of $\widetilde O(1/\sqrt{\epsilon_2})$; (iii) the time complexity of our NCG-B algorithm is at least the same and usually better than their time complexity according to our analysis shown later. Additionally, their conjugate gradient method could fail due to the inexact smallest eigen-value computed by the randomized Lanczos method, and their first-order and second-order convergence guarantee could be on different points. \cite{DBLP:journals/corr/CarmonDHS16} developed an algorithm that utilizes the negative curvature descent to reach a region that is almost convex and then switches to accelerated gradient method to decrease the magnitude of the gradient. Our NCG-B algorithm is built on this development by replacing their negative curvature descent with our NCG-A algorithm, which has the same guarantee on the smallest eigen-value of the returned solution but uses a much less number of Hessian-vector products. 

There are less work on first-order methods for non-convex optimization with second-order convergence guarantee. Noisy stochastic gradient methods can provide such guarantee by adding noise into the stochastic gradient~\citep{pmlr-v40-Ge15,DBLP:conf/colt/ZhangLC17}. However, their time complexity has a polynomial factor of the problem's dimensionality. Recently, \cite{jin2017escape} proposed a noisy gradient method that adds noise into the iterate for evaluating the gradient to escape from the saddle point. To achieve an $(\epsilon, \sqrt{\epsilon})$-second-order stationary point, the noisy gradient method requires $\widetilde O(1/\epsilon^2)$ gradient evaluations.

\section{Preliminaries and Warm-up}\label{sec:pre}
A function $f(\x)$ is smooth if its gradient is Lipschitz continuous, i.e., there exists $L_1>0$ such that $\|\nabla f(\x) - \nabla f(\y)\|\leq L_1\|\x - \y\|$ hold for all $\x, \y$. The Hessian of a twice differentiable function $f(\x)$  is Lipschitz continuous, if there exists $L_2>0$ such that $\|\nabla^2 f(\x) - \nabla^2 f(\y)\|_2\leq L_2\|\x - \y\|$ for all $\x, \y$, where $\|X\|_2$ denotes the spectral norm of a matrix $X$. A function $f(\x)$ is called $\mu$-strongly convex ($\mu>0$) if 
\begin{align}
    f(\y)\geq f(\x) + \nabla f(\x)^{\top}(\y - \x) + \frac{\mu}{2}\|\y - \x\|^2, \forall \x, \y
\end{align}
If $f(\x)$ satisfies the above condition for $\mu<0$, it is referred to as  $\gamma$-almost convex with $\gamma=-\mu$. We denote by $T_g$ and $T_h$ the runtime for computing the gradient of $f(\x)$ and a Hessian-vector product. Throughout the paper, we make the following assumptions.
\begin{ass}For the optimization problem~(\ref{eqn:opt}), we assume 
	\label{ass:1}
\begin{itemize}
\item the objective function $f(\x)$ is twice differentiable, and 
\item it has Lipschitz continuous gradient: $\|\nabla f(\x) - \nabla f(\y)\|\leq L_1\|\x - \y\|$, and 
\item it has Lipschitz continuous Hessian: $\|\nabla^2 f(\x) - \nabla f(\y)\|\leq L_2\|\x - \y\|$, and
\item given an initial solution $\x_0$, there exists $\Delta<\infty$ such that $f(\x_0) - f(\x_*)\leq \Delta$, where $\x_*$ denotes the global minimum of~(\ref{eqn:opt}).
\end{itemize}
\end{ass}
Next, we describe the gradient descent and the negative curvature descent methods, and present their first-order and second-order convergence guarantee, respectively. The gradient descent (GD) method in Algorithm~\ref{alg:gd} is standard and its first-order convergence guarantee is presented below. 
\begin{algorithm}[t]
\caption{Gradient-Descent: GD$(\x_0, \epsilon)$}\label{alg:gd}
\begin{algorithmic}[1]
\STATE \textbf{Input}:  $\x_0$, $\epsilon$
\STATE $\x_1=\x_0$, $\eta = 1/L_1$%, $T_1 \geq 81Lc^2\epsilon_0^{2\theta -1}$
\FOR{$j=1,2,\ldots,$}
\IF{ $\|\nabla f(\x_j)\|>\epsilon$}
\STATE $\x_{j+1} = \x_j - \eta \nabla f(\x_j)$
\ELSE
\STATE return $\x_j$
\ENDIF

%\STATE 
\ENDFOR
\end{algorithmic}
\end{algorithm}
\begin{algorithm}[t]
\caption{Negative Curvature Descent: NCD$(\x_0, \epsilon, \delta)$}\label{alg:ncd}
\begin{algorithmic}[1]
\STATE \textbf{Input}:  $\x_0$, $\epsilon$, $\delta$
\STATE $\x_1=\x_0$, $\delta' = \delta /(1+12L_2^2\Delta/\epsilon^3)$%, $T_1 \geq 81Lc^2\epsilon_0^{2\theta -1}$
\FOR{$j=1,2,\ldots,$}
\STATE Find a vector $\v_j$ such that $\|\v_j\|=1$ and, with probability at least $1-\delta'$,
\[
\lambda_{\min}(\nabla^2 f(\x_j))\geq \v_j^{\top}\nabla^2 f(\x_j)\v_j - \epsilon/2
\]
using the Lanczos method. 
\IF{ $\v_j^{\top}\nabla^2 f(\x_j)\v_j\leq -\epsilon/2$}
\STATE $\x_{j+1} = \x_j - \frac{2|\v_j^{\top}\nabla^2 f(\x_j)\v_j|}{L_2}\text{sign}(\v_j^{\top}\nabla f(\x_j))\v_j$
\ELSE
\RETURN $\x_j$
\ENDIF
%\STATE 
\ENDFOR
\end{algorithmic}
\end{algorithm}
\begin{prop}\label{lemma:1}
The GD algorithm terminates at iteration $j$ for some  $j\leq 1 + \frac{2L_1\Delta}{\epsilon^2}$ with $\|\nabla f(\x_j)\|\leq \epsilon$.
\end{prop}
\begin{comment}\begin{proof}
Due to the smoothness of $f(\x)$, for $\eta\leq 1/L_1$ we have
\begin{align*}
f(\x_{t+1})& \leq f(\x_t) + \nabla f(\x_t)^{\top}(\x_{t+1} - \x_t) + \frac{L_1}{2}\|\x_{t+1} - \x_t\|^2\\
&= f(\x_t) - \eta\|\nabla f(\x_t)\|_2^2 + \frac{L_1\eta^2}{2}\|\nabla f(\x_t)\|^2\\
&\leq f(\x_t) - \frac{\eta}{2}\|\nabla f(\x_t)\|^2
\end{align*}
Let $j_* = \inf_{t}\{t: \|\nabla f(\x_t)\|\leq \epsilon\}$. For any $t< j_*$, we have $\|\nabla f(\x_t)\|> \epsilon$. Then 
\begin{align*}
\sum_{t=1}^{j_* - 1}\frac{\eta\epsilon^2}{2}\leq \sum_{t=1}^{j_* - 1}\frac{\eta}{2}\|\nabla f(\x_t)\|^2\leq f(\x_1) - f(\x_{j_*})\leq \Delta
\end{align*}
Therefore after  $j_*=\left\lceil\frac{2L_1\Delta}{\epsilon^2}+1\right\rceil$ iterations for $\eta=1/L_1$, we have $\|\nabla f(\x_{j_*})\|\leq \epsilon$. 
\end{proof}
\end{comment}
For the negative curvature descent (NCD) method, we consider the algorithm proposed in~\citep{DBLP:journals/corr/CarmonDHS16}. For completeness, the algorithm is presented in Algorithm~\ref{alg:ncd}. The algorithm needs a subroutine to compute the smallest eigen-vector of the Hessian to approximate the corresponding eigen-value up to a prescribed noise level. The Lanczos method with a random starting vector can serve the purpose. When the Hessian matrix has a finite-sum structure, several approximate PCA algorithms can be also used with a reduced time complexity~\citep{DBLP:conf/nips/ZhuL16,DBLP:conf/icml/GarberHJKMNS16}.  Without loss of generality, we will use the Lanczos method for deterministic optimization and consider the approximate PCA algorithms when we discuss the stochastic variant. To this end, we first present a runtime result of the Lanczos method. 
\begin{lemma}\citep{clement17}[Lemma 11]\label{lemma:lanczos}
Let $H\in\R^d$ be a symmetric matrix satisfying $\|H\|_2\leq L_1$.
Suppose that the Lanczos method is applied to find the largest eigenvalue of $L_1I - H$
starting at a random vector uniformly distributed over the unit sphere. Then, for any
$\varepsilon>0$ and $\delta\in (0, 1)$, there is a probability at least $1-\delta$ that the method outputs a
vector $\v$ such that 
$\lambda_{\min}(H)\geq \v^{\top}H\v - \varepsilon$ with at most $\min\left(d, \frac{\log(d/\delta^2)\sqrt{L_1}}{2\sqrt{2\varepsilon}}\right)$ Hessian-vector products.
\end{lemma}
{\bf Remark:} In the sequel, we assume $d$ is very large and the time complexity of the Lanczos method for finding an approximate negative curvature of the Hessian up to a noise level $\varepsilon$ is given by $O(\sqrt{L_1}\log(d/\delta)/\sqrt{\varepsilon})$.

With the above Lemma, the guarantee of NCD is given in Proposition \ref{lemma:2}, which corresponds to Lemma 3.2 in~\citep{DBLP:journals/corr/CarmonDHS16}.

\begin{prop}
\label{lemma:2}
The NCD algorithm terminates at iteration $j$ for some 
\[
j\leq 1 + \frac{12L_2^2(f(\x_1) - f(\x_j))}{\epsilon^3}\leq 1 + \frac{12L_2^2\Delta}{\epsilon^3}
\] and with probability at least $1-\delta$
\[
\lambda_{\min}(\nabla^2 f(\x_j))\geq -\epsilon
\]
Furthermore each iteration requires time at most 
\[
O\left(T_h\sqrt{\frac{L_1}{\epsilon}}\log\left(\frac{d}{\delta}(1+12L_2^2\Delta/\epsilon^3\right)\right)
\]
\end{prop}

\section{The NCG Step and the NCG-A Algorithm}\label{sec:NCG}
In this section, we present a novel updating step that lets the noisy NCD step to compete with  the GD step. The idea is intuitive by observing the fact that following the GD step  the  gradient can converge to zero and following the NCD method the smallest eigen-value of the Hessian can converge to non-negative orthant according to the results in Proposition~\ref{lemma:1} and Proposition~\ref{lemma:2}. Hence, it is natural to combine the GD step and NCD step in a unified algorithm in order to provide convergence guarantee on the gradient and the smallest eigen-value of the Hessian simultaneously. To this end, we propose a novel NCG step shown in Algorithm~\ref{alg:gnc}. First, we compute a noisy normalized negative curvature $\v$ that approximates the smallest eigen-value of the Hessian at the current point $\x$ up to a noise level $\varepsilon$, i.e., $\lambda_{\min}(\nabla^2 f(\x))\geq \v^{\top}\nabla^2 f(\x)\v - \varepsilon$. Then we take either the noisy negative curvature direction or the negative gradient direction depending on which decreases the objective value more. This is done by comparing the estimations of the objective decrease for following these two directions as shown in Step 3 in Algorithm~\ref{alg:gnc}. We refer to the proposed updating step as the NCG step short for noisy negative curvature descent competing with gradient descent. The fundamental difference between the NCG step and the updating steps appearing in other works~\citep{curtis2017negative,clement17,DBLP:journals/corr/CarmonDHS16} is that it allows for a free parameter $\varepsilon$ specifying the noisy level for computing the negative curvature. It is this property and the competing idea that enable us to develop improved algorithms for non-convex optimization with a reduced number of Hessian-vector products. Please note that the time complexity of a NCG step  is given by $O\left(T_h\sqrt{L_1/\varepsilon}\log(d/\delta)\right)$.

%\begin{figure}[t]
%\centering
%\framebox[1.\textwidth]{
\begin{algorithm}[t]
\caption{The NCG Step: $\left(\x^+,\v\right)=\text{NCG}(\x, \varepsilon, \delta)$}\label{alg:gnc}
\begin{algorithmic}[1]
\STATE \textbf{Input}:  $\x$, $\varepsilon$, $\delta$
\STATE Find a vector $\v$ such that $\|\v\|=1$ and, with probability at least $1-\delta$,
\[
\lambda_{\min}(\nabla^2 f(\x))\geq \v^{\top}\nabla^2 f(\x)\v - \varepsilon
\]
using the Lanczos method. 
\IF{$\frac{2|\v^{\top}\nabla^2 f(\x)\v|^3}{3L_2^2}>\frac{\|\nabla f(\x)\|^2}{2L_1}$}
\STATE Compute $\x^+  = \x - \frac{2|\v^{\top}\nabla^2 f(\x)\v|}{L_2}\text{sign}(\v^{\top}\nabla f(\x))\v$
\ELSE
\STATE Compute $\x^+ = \x - \frac{1}{L_1}\nabla f(\x)$
\ENDIF
\RETURN $\x^+,\v$
%\STATE 
\end{algorithmic}
\end{algorithm}
\begin{algorithm}[t]
\caption{NCG-A1: $(\x_0, \epsilon_1, \epsilon_2, \delta)$}\label{alg:gnc-a}
\begin{algorithmic}[1]
\STATE \textbf{Input}:  $\x_0$, $\epsilon_1, \epsilon_2$, $\delta$
\STATE $\x_1=\x_0$, $\delta' = \delta /(1+\max\left(\frac{12L_2^2}{\epsilon_2^3}, \frac{2L_1}{\epsilon_1^2}\right)\Delta)$%, $T_1 \geq 81Lc^2\epsilon_0^{2\theta -1}$
\FOR{$j=1,2,\ldots,$}
\STATE  $(\x_{j+1},\v_j) = \text{NCG}(\x_j,  \max(\epsilon_2, \|\nabla f(\x_j)\|)/2,  \delta')$

%Find a vector $\v_j$ such that $\|\v_j\|=1$ and, with probability at least $1-\delta'$,
%\[
%\lambda_{\min}(\nabla^2 f(\x_j))\geq \v_j^{\top}\nabla^2 f(\x_j)\v_j - \max(\epsilon_2, \|\nabla f(\x_j)\|)/2
%\]
%using a leading eigen-vector computation. 
\IF{ $\v_j^{\top}\nabla^2 f(\x_j)\v_j> -\epsilon_2/2$ and $\|\nabla f(\x_j)\|\leq \epsilon_1$}
\RETURN $\x_j$
%\ELSIF{$\frac{2|\v_j^{\top}\nabla f^2(\x_j)\v_j|}{3L_2^2}>\frac{\|\nabla f(\x_)\|^2}{2L_1}$}
%\STATE Compute $\x_{j+1}  = \x_j - \frac{2|\v_j^{\top}\nabla^2 f(\x_j)\v_j|}{L_2}sign(\v_j^{\top}\nabla f(\x_j)\v_j)\v_j$
%\ELSE
%\STATE Compute $\x_{j+1} = \x_j - \frac{1}{L_1}\nabla f(\x_j)$
\ENDIF
%\STATE 
\ENDFOR
\end{algorithmic}
\end{algorithm}
%\label{fig:1}
%\vspace{-0.2in}
%\end{figure}

Next, we present the first variant of our first algorithm employing the NCG step, referred to as the NCG-A1  shown in Algorithm~\ref{alg:gnc-a}. The noise parameter for each call of NCG depends on the magnitude of the current gradient. In particular, for the prescribed accuracy level for the first-order and second-order guarantee $\epsilon_1, \epsilon_2\geq \epsilon_1$, the noise parameter of the NCG step at $t$-th iteration is $\max(\epsilon_2, \|\nabla f(\x_t)\|)/2$, which could be much larger than $\epsilon_2/2$ before reaching a neighborhood of a first-order critical point. The  complexity guarantee  of NCG-A1 is presented in the following theorem.

\begin{theorem}
\label{lemma:GNC-A}
The NCG-A1 algorithm terminates at iteration $j_*$ for some 
\begin{equation}
\label{eqn:bound-GNC-A}
j_*\leq 1 + \max\left(\frac{12L_2^2}{\epsilon_2^3}, \frac{2L_1}{\epsilon_1^2}\right)(f(\x_1) - f(\x_{j_*}))\leq 1 +  \max\left(\frac{12L_2^2}{\epsilon_2^3}, \frac{2L_1}{\epsilon_1^2}\right)\Delta,
\end{equation}
with $\|\nabla f(\x_{j_*})\|\leq \epsilon_1$, and with probability at least $1-\delta$
\[
\lambda_{\min}(\nabla^2 f(\x_{j_*}))\geq -\max\left(\epsilon_2,  \epsilon_1 \right)
\]
Furthermore, the $j$-th iteration requires time at most 
\[
O\left(T_h\frac{\sqrt{L_1}}{\max(\epsilon_2, \|\nabla f(\x_j)\|)^{1/2}}\log\left(\frac{d}{\delta'}\right)\right)
\]
\end{theorem}
{\bf Remark:} Let us make a remark of the result above. First, when $\epsilon_2 = \epsilon_1= \epsilon$, the iteration complexity of NCG-A1 for achieving a point with  $\max\{\|\nabla f(\x)\|, -\lambda_{\min}(\nabla^2 f(\x))\}\leq \epsilon$ is $O(1/\epsilon^3)$, which match the results in previous works (e.g.~\citep{clement17,Cartis2011,Cartis2011b,peng16inexacthessian}). However, the number of Hessian-vector products in NCG-A1 could be much less than that in these existing works. For example, the number of Hessian-vector products in~\citep{clement17,peng16inexacthessian} is $\widetilde O(1/\sqrt{\epsilon_2})$ at each iteration requiring the second-order information. Second, the worse-case time complexity of NCG-A1  is given by $\widetilde O\left(T_h\max\left\{\frac{1}{\epsilon_1^2\epsilon_2^{1/2}}, \frac{1}{\epsilon_2^{7/2}}\right\}\right)$  using the worse-case time complexity of each iteration, which is the same as the result of Theorem 2 in~\citep{peng16inexacthessian}~\footnote{When we compare with~\citep{peng16inexacthessian}, we consider the variant of NCG-A1 with an inexact Hessian presented later.}.  

One might notice that if we plug  $\epsilon_1=\epsilon, \epsilon_2 = \sqrt{\epsilon}$ into the worst-case time complexity of NCG-A1, we end up with $\widetilde O(1/\epsilon^{9/4})$, which is worse the best time complexity $\widetilde O(1/\epsilon^{7/4})$ found in literature (e.g., \citep{DBLP:conf/stoc/AgarwalZBHM17,DBLP:journals/corr/CarmonDHS16}). We will improve the worst-case time complexity in Section~\ref{sec:NCGc} for $\epsilon_2> \epsilon_1$. 
%If we set $\epsilon_1 = \epsilon^{3/4}$ and $\epsilon_2=\epsilon^{1/2}$, GNC-A uses at most $\widetilde O(\epsilon^{-7/4})$ complexity to achieve a solution $\x_j$ such that $\|\nabla f(\x_j)\|\leq \epsilon^{3/4}$ and $\lambda_{\min}(\nabla^2 f(\x_j))\geq - \sqrt{\epsilon}$ with high probability. If we set $\epsilon_1 = \epsilon$ and $\epsilon_2 = \epsilon^{2/3}$, then GNC-A uses at most  $\widetilde O(\epsilon^{-7/3})$ complexity to achieve a solution $\x_j$ such that $\|\nabla f(\x_j)\|\leq \epsilon$ and $\lambda_{\min}(\nabla^2 f(\x_j))\geq - \epsilon^{2/3}$ with high probability. If we set $\epsilon_1 = \epsilon$ and $\epsilon_2 = \epsilon$, then GNC-A uses at most $\widetilde O(\epsilon^{-7/2})$  to achieve a solution $\x_j$ such that $\|\nabla f(\x_j)\|\leq \epsilon$ and $\lambda_{\min}(\nabla^2 f(\x_j))\geq - \epsilon$. 

To prove Theorem~\ref{lemma:GNC-A}, we first present a lemma showing that each NCG step will decrease the objective value.
\begin{lemma}\label{lem:NCG}
For any $\varepsilon, \delta>0$,  the update $\x_{j+1} = \text{NCG}(\x_j,  \varepsilon,  \delta)$ yields  
\begin{align*}
f(\x_j) - f(\x_{j+1})\geq \max\left( \frac{2|\v_j^{\top}\nabla ^2 f(\x_j)\v_j|^3}{3L_2^2}, \frac{\|\nabla f(\x_j)\|^{2}}{2L_1}\right)
\end{align*}
\end{lemma}

 \vspace*{0.2in}
\begin{proof}[Proof of Theorem~\ref{lemma:GNC-A}]
Let $j_*$ denote the $j$ such that the algorithm terminates. Then for all $j<j_*$, we have $\|\nabla f(\x_j)\|> \epsilon_1$, or $\v_j^{\top}\nabla^2 f(\x_j)\v_j\leq  - \epsilon_2/2$. 
\begin{comment}Let $\x_{j+1}^1$ be the updated solution according to the NCD update and $\x_{j+1}^2$ be the updated solution according to the GD update. We can prove that
\begin{align}
f(\x_{j+1}^1) - f(\x_j)\leq - \frac{2|\v_j^{\top}\nabla ^2 f(\x_j)\v_j|^3}{3L_2^2} = - \Delta_j^1\label{eqn:decrease1}\\
f(\x_{j+1}^2) - f(\x_j) \leq - \frac{1}{2L_1}\|\nabla f(\x_j)\|^2 = - \Delta_j^2
\end{align}
where the inequality (\ref{eqn:decrease1}) follows from the $L_2$-Lipschitz continuity of the Hessian.
If $\Delta_j^1>\Delta_j^2$, then $f(\x_{j+1}) - f(\x_j) =  f(\x_{j+1}^1) - f(\x_j)\leq -\Delta_j^1$. If $\Delta_j^2>\Delta^j_1$, then  $f(\x_{j+1}) - f(\x_j) =  f(\x_{j+1}^2) - f(\x_j)\leq -\Delta_j^2$. In both cases, we have
\begin{align*}
\max(\Delta_j^1, \Delta_j^2)\leq f(\x_j ) - f(\x_{j+1})
\end{align*}
\end{comment}
According to Lemma~\ref{lem:NCG}, we have
\begin{align*}
f(\x_j) - f(\x_{j+1})\geq \max\left( \frac{2|\v_j^{\top}\nabla ^2 f(\x_j)\v_j|^3}{3L_2^2}, \frac{\|\nabla f(\x_j)\|^{2}}{2L_1}\right)
\end{align*}
Let us consider three cases. Case 1: $\|\nabla f(\x_j)\|> \epsilon_1$ and $\v_j^{\top}\nabla^2 f(\x_j)\v_j\leq  -\epsilon_2/2$, then we have
\begin{align*}
\max\left(\frac{\epsilon_2^3}{12L_2^2}, \frac{\epsilon_1^{2}}{2L_1}\right)\leq f(\x_j) - f(\x_{j+1})
\end{align*}
Case 2:  $\|\nabla f(\x_j)\|\leq\epsilon_1$ and $\v_j^{\top}\nabla^2 f(\x_j)\v_j\leq  -\epsilon_2/2$, we have
\begin{align*}
\frac{\epsilon_2^3}{12L_2^2}\leq f(\x_j) - f(\x_{j+1})
\end{align*}
Case 3:  $\|\nabla f(\x_j)\|> \epsilon_1$ and $\v_j^{\top}\nabla^2 f(\x_j)\v_j>  - \epsilon_2/2$, we have
\begin{align*}
\frac{\epsilon_1^2}{2L_1}\leq f(\x_j) - f(\x_{j+1})
\end{align*}
In any case, we have
\begin{align*}
\min\left(\frac{\epsilon_1^2}{2L_1}, \frac{\epsilon_2^3}{12L_2^2}\right)\leq f(\x_j) - f(\x_{j+1})
\end{align*}
Then with at most $j_* = 1 +  \max\left(\frac{12L_2^2}{\epsilon_2^3}, \frac{2L_1}{\epsilon_1^2}\right)\Delta$, the algorithm terminates. Upon termination, we have with probability at least $1-j_*\delta'=1-\delta$, 
\[
\lambda_{\min}(\nabla^2 f(\x_{j_*}))\geq -\epsilon_2/2 - \max(\epsilon_2, \|\nabla f(\x_{j_*})\|)/2.
\]
Since we have $\max(\epsilon_2, \|\nabla f(\x_{j_*})\|) \leq \max(\epsilon_2, \epsilon_1)$,  therefore $\lambda_{\min}(\nabla^2 f(\x_{j_*}))\geq -\epsilon_2/2 - \max(\epsilon_2, \epsilon_1)/2$. If $\epsilon_1>\epsilon_2$, we have $\lambda_{\min}(\nabla^2 f(\x_{j_*}))\geq - \frac{\epsilon_1 + \epsilon_2}{2}\geq - \epsilon_1$. Otherwise, we have $\lambda_{\min}(\nabla^2 f(\x_{j_*}))\geq - \epsilon_2$.

The running time spent on the $j$-th iteration follows from Lemma~\ref{lemma:lanczos}.

\end{proof}

\subsection{NCG-A2 with even reduced per-iteration complexity}\label{sec:NCGb}
In this subsection and next section, we focus on improving the time complexity of NCG-A1 for $\epsilon_2>\epsilon_1$ trying to further reduce the per-iteration cost and to match the best time complexity in literature for $\epsilon_2 = \sqrt{\epsilon_1}$.  Indeed, we will consider all values of $\epsilon_2\in[\epsilon_1, 1)$. To this end, we characterize $\epsilon_2 = \epsilon_1^\alpha$ with $\alpha\in(0, 1]$. In this subsection, we will present  a variant of NCG-A1 having a lower time complexity for each iteration. In particular, the time complexity of each iteration is given by $\widetilde O(1/\sqrt{\max(\epsilon_2, \|\nabla f(\x_t)\|^\alpha)})$ for $\alpha\in(0, 1]$. 
%that achieves the best time complexity for $\alpha\in[4/7, 1]$ and also improves that of NCG-A for $\alpha\in[4/7, 2/3]$. In the next section, we present the third algorithm that further improves the worst-case time complexity of NCG-B for $\alpha\in[1/2, 4/7]$. 
The variant NCG-A2 is presented %in Algorithm~\ref{alg:gnc-a2modified}, which calls a subroutine for each iteration shown 
in Algorithm~\ref{alg:gnc-b}, which is very similar to NCG-A1 with slight modification on the noise parameter of each NCG step.
\begin{algorithm}[t]
\caption{NCG-A2: $(\x_0, \epsilon_1, \alpha, \delta)$}\label{alg:gnc-b}
\begin{algorithmic}[1]
\STATE \textbf{Input}:  $\x_0$, $\epsilon_1, \alpha$, $\delta$
\STATE $\x_1=\x_0$, $\epsilon_2 = \epsilon_1^\alpha$,  $\delta' = \delta /(1+\max\left(\frac{12L_2^2}{\epsilon_2^3}, \frac{2L_1}{\epsilon_1^2}\right)\Delta)$, %, $T_1 \geq 81Lc^2\epsilon_0^{2\theta -1}$
\FOR{$j=1,2,\ldots,$}
\STATE  $(\x_{j+1},\v_j) = \text{NCG}(\x_j,  \max(\epsilon_2, \|\nabla f(\x_j)\|^\alpha)/2,  \delta')$

%Find a vector $\v_j$ such that $\|\v_j\|=1$ and, with probability at least $1-\delta'$,
%\[
%\lambda_{\min}(\nabla^2 f(\x_j))\geq \v_j^{\top}\nabla^2 f(\x_j)\v_j - \max(\epsilon_2, \|\nabla f(\x_j)\|)/2
%\]
%using a leading eigen-vector computation. 
\IF{ $\v_j^{\top}\nabla^2 f(\x_j)\v_j> -\epsilon_2/2$ and $\|\nabla f(\x_j)\|\leq \epsilon_1$}
\RETURN $\x_j$
%\ELSIF{$\frac{2|\v_j^{\top}\nabla f^2(\x_j)\v_j|}{3L_2^2}>\frac{\|\nabla f(\x_)\|^2}{2L_1}$}
%\STATE Compute $\x_{j+1}  = \x_j - \frac{2|\v_j^{\top}\nabla^2 f(\x_j)\v_j|}{L_2}sign(\v_j^{\top}\nabla f(\x_j)\v_j)\v_j$
%\ELSE
%\STATE Compute $\x_{j+1} = \x_j - \frac{1}{L_1}\nabla f(\x_j)$
\ENDIF
%\STATE 
\ENDFOR
\end{algorithmic}
\end{algorithm}

\begin{theorem}
\label{thm:gnc-a2:iteration}
	For any $\alpha\in(0, 1]$, the NCG-A2 algorithm terminates at iteration $j_*$ for some 
\begin{equation}
\label{eqn:bound-GNC-B}
j_*\leq 1 + \max\left(\frac{12L_2^2}{\epsilon_1^{3\alpha}}, \frac{2L_1}{\epsilon_1^2}\right)(f(\x_1) - f(\x_{j_*}))\leq 1 +  \max\left(\frac{12L_2^2}{\epsilon_1^{3\alpha}}, \frac{2L_1}{\epsilon_1^2}\right)\Delta,
\end{equation}
with $\|\nabla f(\x_{j_*})\|\leq \epsilon_1$, and with probability at least $1-\delta$
\[
\lambda_{\min}(\nabla^2 f(\x_{j_*}))\geq -\epsilon_1^\alpha
\]
Furthermore, the $j$-th iteration requires time at most 
\[
O\left(T_h\frac{\sqrt{L_1}}{\max(\epsilon_1^\alpha, \|\nabla f(\x_j)\|^\alpha)^{1/2}}\log\left(\frac{d}{\delta'}\right)\right)
\]
\end{theorem}

\begin{comment}
\begin{theorem}
\label{thm:gnc-a2:iteration}
	For any $\alpha\in[1/2, 1]$, with probability at least $1-\delta$, the Algorithm NCG-B returns a vector $\x_j$ such that $\|\nabla f(\x_j)\|\leq\epsilon_1$ and $\lambda_{\text{min}}(\nabla^2 f(\x_j))\geq -\epsilon_1^\alpha$ in at most $O\left(\frac{L_1\Delta}{\epsilon_1^2} +\frac{L_2^2\sqrt{L_1}\Delta}{\epsilon_1^{3\alpha}}\right)$ total iterations (including all iteration within and outside NCG-Bsub), and each iteration $j$ within NCG-Bsub requires time at most  
	\[
	O\left(T_h\frac{\sqrt{L_1}}{\max(\epsilon_1^\alpha, \|\nabla f(\x_j)\|^\alpha)^{1/2}}\log\left(\frac{d}{\delta'}\right)\right)
	\]
and the worse-case time complexity of NCG-B is $\widetilde O\left(\frac{L_1\Delta}{\epsilon_1^2}T_g +\frac{L_2^2L_1\sqrt{L_1}\Delta}{\epsilon_1^{7\alpha/2}}T_h\right)$
\end{theorem}
\end{comment}
{\bf Remark: } Let $\epsilon_2= \epsilon_1^\alpha$.  First, the algorithm finds a solution satisfying $\|\nabla f(\x_j)\|\leq \epsilon_1$ and $\lambda_{\min}(\nabla ^2 f(\x_j))\geq - \epsilon_2$ with high probability using $O\left(\frac{1}{\epsilon_1^2} + \frac{1}{\epsilon_2^3}\right)$ iterations, which is the same iteration complexity as NCG-A1. The difference is that each iteration of NCG-A2 will require a less number of Hessian-vector products when $\alpha<1$ and $\|\nabla f(\x_j)\|\leq 1$. This is because that the time complexity of each iteration is $\widetilde O\left(T_h\frac{\sqrt{L_1}}{\max(\epsilon_2, \|\nabla f(\x_j)\|^\alpha)^{1/2}}\right)$, which could be much less than $\widetilde O\left(T_h\frac{\sqrt{L_1}}{\max(\epsilon_2, \|\nabla f(\x_j)\|)^{1/2}}\right)$ for each iteration of NCG-A1 when $\alpha<1$ and $\|\nabla f(\x_j)\|\leq 1$. The proof of the above Theorem is very similar to that of Theorem~\ref{lemma:GNC-A} and is included in the appendix.  %The difference is that not every iteration in NCG-B requires computing the negative curvature. Only the iterations within NCG-Bsub requires computing a noisy negative curvature, which yields the worse-case time complexity of $\widetilde O\left(\frac{L_1\Delta}{\epsilon_1^2}T_g +\frac{L_2^2L_1\sqrt{L_1}\Delta}{\epsilon_2^{7/2}}T_h\right)$, which improves that of NCG-A. Second, considering that $T_g\in O(T_h)$, for $\alpha\in[2/3,1]$ both NCG-A and NCG-B have the same order of worse-case time complexity. It is when $\alpha<2/3$ that the worse-case complexity of NCG-B dominated by $\widetilde O(1/\epsilon^{7\alpha/2})$ becomes  better than that of NCG-A dominated by $\widetilde O(1/\epsilon^{2+\alpha/2})$. Nevertheless, the worse-case time complexity of NCG-B for $\alpha=1/2$ is $\widetilde O(1/\epsilon^{2})$, which is still worse than the best time complexity in literature. Third, as a matter of fact with the worse-case time complexity of $\widetilde O(1/\epsilon^{2})$, NCG-B can find a $(\epsilon, \epsilon^{4/7})$-second-order stationary  solution by setting $\alpha = 4/7$. This result is stronger than that\ in \citep{DBLP:journals/corr/CarmonD16}, which adopts a gradient descent method with Hessian-vector products for solving the cubic regularization step and suffers from  a time complexity of $\widetilde O(1/\epsilon^{2})$ for achieving a $(\epsilon, \epsilon^{1/2})$-second-order stationary solution. 

\begin{algorithm}[t]
	\caption{NCG-B1: $(\x_0, \epsilon_1, \epsilon_2, \delta)$}\label{alg:gnc-a3}
	\begin{algorithmic}[1]
		\STATE \textbf{Input}:  $\x_0$, $\epsilon_1$, $\epsilon_2$, $\delta$
		\STATE $K:=\lceil1+\Delta\left(\frac{\max(12L_2^2, 2L_1)}{\epsilon_2^3}+\frac{2\sqrt{10}L_2}{\epsilon_1\epsilon_2}\right)\rceil$, $\delta':=\delta/K$, %where $\epsilon_2^{3/2}\leq \epsilon_1'\leq\epsilon_2<1$
		\FOR{$k=1,2,\ldots,$}
		\STATE $\xh_k = \text{NCG-A1}(\x_k, \epsilon_2^{3/2}, \epsilon_2, \delta')$
		\IF{$\|\nabla f(\xh_k)\|\leq  \epsilon_1$}
		\RETURN $\xh_k$
		\ELSE
		\STATE  Set $f_k(\x) = f(\x) + L_1\left([\|\x - \xh_j\| - \epsilon_2/L_2]_+\right)^2$
		\STATE $\x_{k+1} = \text{Almost-Convex-AGD}(f_k, \xh_k, \epsilon_1/2, 3\epsilon_2, 5L_1)$
		\ENDIF
		
		%Find a vector $\v_j$ such that $\|\v_j\|=1$ and, with probability at least $1-\delta'$,
		%\[
		%\lambda_{\min}(\nabla^2 f(\x_j))\geq \v_j^{\top}\nabla^2 f(\x_j)\v_j - \max(\epsilon_2, \|\nabla f(\x_j)\|)/2
		%\]
		%using a leading eigen-vector computation. 
		
		%\STATE 
		\ENDFOR
	\end{algorithmic}
\end{algorithm}

\begin{algorithm}[t]
	\caption{NCG-B2: $(\x_0, \epsilon_1, \alpha, \delta)$}\label{alg:gnc-a4}
	\begin{algorithmic}[1]
		\STATE \textbf{Input}:  $\x_0$, $\epsilon_1$, $\alpha$, $\delta$
		\STATE $\epsilon_2=\epsilon_1^\alpha$, $K:=\lceil1+\Delta\left(\frac{\max(12L_2^2, 2L_1)}{\epsilon_2^3}+\frac{2\sqrt{10}L_2}{\epsilon_1\epsilon_2}\right)\rceil$, $\delta':=\delta/K$, %where $\epsilon_2^{3/2}\leq \epsilon_1'\leq\epsilon_2<1$
		\FOR{$k=1,2,\ldots,$}
		\STATE $\xh_k = \text{NCG-A2}(\x_k, \epsilon_1^{3\alpha/2}, \frac{2}{3}, \delta')$
		\IF{$\|\nabla f(\xh_k)\|\leq  \epsilon_1$}
		\RETURN $\xh_k$
		\ELSE
		\STATE  Set $f_k(\x) = f(\x) + L_1\left([\|\x - \xh_k\| - \epsilon_2/L_2]_+\right)^2$
		\STATE $\x_{k+1} = \text{Almost-Convex-AGD}(f_j, \xh_k, \epsilon_1/2, 3\epsilon_2, 5L_1)$
		\ENDIF
		
		%Find a vector $\v_j$ such that $\|\v_j\|=1$ and, with probability at least $1-\delta'$,
		%\[
		%\lambda_{\min}(\nabla^2 f(\x_j))\geq \v_j^{\top}\nabla^2 f(\x_j)\v_j - \max(\epsilon_2, \|\nabla f(\x_j)\|)/2
		%\]
		%using a leading eigen-vector computation. 
		
		%\STATE 
		\ENDFOR
	\end{algorithmic}
\end{algorithm}
\section{NCG-B: alternating NCG steps and accelerated GD steps}\label{sec:NCGc}
In this section, we present the second algorithm NCG-B that matches the best time complexity in literature for finding an $(\epsilon, \sqrt{\epsilon})$-second-order stationary point. The algorithm is built on the technique in~\citep{DBLP:journals/corr/CarmonDHS16}, i.e., alternatively finding a solution around which the function is almost convex using the NCD algorithm (i.e., the smallest eigen-value of the Hessian is small) and using accelerated gradient method for a slightly perturbed almost convex function to decrease the magnitude of the gradient. The key difference is that we use NCG-A1 or NCG-A2 instead of the NCD algorithm to reach a region where the function is locally almost convex. Therefore, we have two variants of NCG-B that employ NCG-A1 or NCG-A2 as a subroutine  finding a solution around which the function is almost convex. The first variant NCG-B1 is presented in Algorithm~\ref{alg:gnc-a3}.  The procedure \text{Almost-Convex-AGD} is the same as in \citep{DBLP:journals/corr/CarmonDHS16}. For completeness, we present it in the appendix. The convergence guarantee is presented below. 
\begin{theorem}
\label{thm:gnc-a3:iteration}
	With probability at least $1-\delta$, the Algorithm NCG-B1 returns a vector $\xh_k$ such that $\|\nabla f(\xh_k)\|\leq\epsilon_1$ and $\lambda_{\text{min}}(\nabla^2 f(\xh_k))\geq -\epsilon_2$ with at most $O\left(\frac{1}{\epsilon_2^3}+\frac{1}{\epsilon_1\epsilon_2}\right)$ NCG steps in NCG-A1 and 
	$\widetilde O\left[\left(\frac{1}{\epsilon_2^{7/2}}+\frac{1}{\epsilon_1\epsilon_2^{3/2}}\right)+\frac{\epsilon_2^{1/2}}{\epsilon_1^2}\right]$
    gradient steps in Almost-Convex-AGD, and each step $j$ within NCG-A1 requires time at most  
	\[
	O\left(T_h\frac{\sqrt{L_1}}{\max(\epsilon_2, \|\nabla f(\x_j)\|)^{1/2}}\log\left(\frac{d}{\delta'}\right)\right)
	\]
and the worse-case time complexity of NCG-B1 is $\widetilde O\left(\left(\frac{1}{\epsilon_1\epsilon_2^{3/2}} +\frac{1}{\epsilon_2^{7/2}}\right) T_h+\frac{\epsilon_2^{1/2}}{\epsilon_1^2}T_g\right)$
\end{theorem}
{\bf Remark:} First, we note that the subroutine $\text{NCG-A1}(\x_j, \epsilon_2^{3/2}, \epsilon_2, \delta')$ provides the same guarantee as the NCD in~\citep{DBLP:journals/corr/CarmonDHS16} (see Corollary~\ref{cor:GNC-A}), i.e., returning a solution  $\xh_j$ satisfying $\lambda_{\min}(\nabla^2 f(\xh_j))\geq -\epsilon_2$ with high probability. Therefore the proof of the above theorem  follows similar analysis as in~\citep{DBLP:journals/corr/CarmonDHS16}. For completeness, we present a proof in the appendix.  Second, the number of iterations within NCG-A1 is the same as that in NCD employed by~\citep{DBLP:journals/corr/CarmonDHS16}, and the number of iterations within Almost-Convex-AGD is the same as that in~\citep{DBLP:journals/corr/CarmonDHS16}. The improvement of NCG-B1 over \citep{DBLP:journals/corr/CarmonDHS16}'s algorithm is brought by reducing the number of Hessian-vector products for performing each iteration of NCG-A1. Third, when $\epsilon_2\leq \sqrt{\epsilon_1}$, the worst-case time complexity of NCG-B1 is $\widetilde O\left(\frac{1}{\epsilon_1\epsilon_2^{3/2}} +\frac{1}{\epsilon_2^{7/2}}\right) T_h$. In particular, for $\epsilon_2=\sqrt{\epsilon_1}$ it reduces to $\widetilde O(T_h/\epsilon^{7/4})$, which matches the best time complexity in previous results. Finally, we note that NCG-B1 has the same worse-case time complexity as NCG-A1 for $\epsilon_2\in[\epsilon_1, \epsilon_1^{2/3}]$, but improves over NCG-A1 for $\epsilon_2\in[ \epsilon_1^{2/3}, \epsilon_1^{1/2}]$.% Nevertheless, the runtime of each iteration requiring the negative curvature is not directly comparable for these two algorithms. 

%This section devotes to analyzing another variant that alternate GNC steps and AGD-steps for almost convex functions motivated by  \citep{accelerated}. The algorithm is presented in Algorithm~\ref{alg:gnc-a3}, which is almost identical to that in \citep{accelerated} except that the NCD is replaced by GCN-A with  a particular instance of $\epsilon_1 = \epsilon^{3/4}$ and $\epsilon_2=\epsilon^{1/2}$. Note that $\text{GNC-A}(\x_j, \epsilon^{3/4}, \epsilon^{1/2}, \delta')$ has the same termination guarantee as NCD but has arguably lower time complexity. The same analysis yields that with at most $\widetilde O(1/\epsilon^{7/4})$ we achieve a solution $\|\nabla f(\x)\|\leq \epsilon$ and $\lambda_{\min}(\nabla^2 f(\x))\geq - \sqrt{\epsilon}$. No we provide the analysis for any accuracy levels on first-order and second-order optimality guarantees, i.e. $\|\nabla f(\x)\|\leq \epsilon_1$ and $\lambda_{\min}(\nabla^2 f(\x))\geq -\epsilon_2$. 

The second variant of NCG-B2 is presented in Algorithm~\ref{alg:gnc-a4}, employing NCG-A2 as the subroutine. Its convergence guarantee is presented below.
\begin{theorem}
\label{thm:gnc-a3':iteration}
	For any $\alpha\in(0, 1]$, let $\epsilon_2 = \epsilon_1^{\alpha}$.	With probability at least $1-\delta$, the Algorithm NCG-B2 returns a vector $\xh_k$ such that $\|\nabla f(\xh_k)\|\leq\epsilon_1$ and $\lambda_{\text{min}}(\nabla^2 f(\xh_k))\geq -\epsilon_2$ with at most $O\left(\frac{1}{\epsilon_2^3}+\frac{1}{\epsilon_1\epsilon_2}\right)$ NCG steps in NCG-A2 and  
	$\widetilde O\left[\left(\frac{1}{\epsilon_2^{7/2}}+\frac{1}{\epsilon_1\epsilon_2^{3/2}}\right)+\frac{\epsilon_2^{1/2}}{\epsilon_1^2}\right]$
    gradient steps
    in Almost-Convex-AGD, and each step $j$ within NCG-A2 requires time at most  
	\[
	O\left(T_h\frac{\sqrt{L_1}}{\max(\epsilon_2, \|\nabla f(\x_j)\|^{2/3})^{1/2}}\log\left(\frac{d}{\delta'}\right)\right)
	\]
and the worse-case time complexity of NCG-B2 is $\widetilde O\left(\left(\frac{1}{\epsilon_1\epsilon_2^{3/2}} +\frac{1}{\epsilon_2^{7/2}}\right) T_h+\frac{\epsilon_2^{1/2}}{\epsilon_1^2}T_g\right)$
\end{theorem}
{\bf Remark:} The improvement of NCG-B2 over NCG-B1 lies at the reduced per-iteration cost for each NCG step when $\|\nabla f(\x_j)\|\leq 1$.

\section{Extensions and Discussions}\label{sec:ext}
In this section, we will present variants of NCG-based algorithms that use inexact Hessian at each iteration requiring the second-order information and establish their complexities same as the variants with exact Hessian. We also present a pure stochastic algorithm and estabilsh its convergence guarantee. Detailed comparisons with two recent works~\citep{peng16inexacthessian,clement17} are also provied.  Finally, we discuss the applications of the proposed algorithms for solving some strict-saddle functions and their time complexities.

%compare different algorithms in terms of both theory and application. We also analyze one extension of our developed schemes in the finite sum circumstance, by substituting the Hessian by the subsampled Hessian.

\subsection{Variants of NCG-based algorithms with inexact Hessian}
To derive the same complexity results of the variants with inexact Hessian, we assume a regularity condition on the inexact Hessian $H(\x_j)$ at any point $\x_j$. 
\begin{ass}\label{ass:iH}
For any $\epsilon_2\in(0, 1)$, assume there exists $M<\infty, \epsilon_3\leq \epsilon_2/12$ such that at any points $\x_j\in\R^d$ we have
\begin{align}
    \|(H(\x_j) - \nabla^2 f(\x_j))\|_2&\leq \epsilon_3\label{eqn:ine}\\
    \|H(\x_j)\|&\leq M.
\end{align}
\end{ass}
One might notice that  the first condition seems to be stringent than that in~\citep{peng16inexacthessian}, i.e., $\|(H(\x_j) - \nabla^2 f(\x_j))\s_t\|\leq \epsilon_3\|\s_t\|$. Nevertheless, we aim to establish a stronger second-order convergence result. In particular, we prove that $\lambda_{\min}(\nabla^2 f(\x_j))\geq -\epsilon_2$ upon termination, while the result in~\citep{peng16inexacthessian} only guarantees that $\lambda_{\min}(H(\x_j))\geq -\epsilon_2$ upon termination of their algorithms. It seems that the condition in~(\ref{eqn:ine}) is inevitable to ensure $\lambda_{\min}(\nabla^2 f(\x_j))\geq -\epsilon_2$ for the returned solution. 

The variants of the NCG-step and the NCG-A1 algorithm employing the inexact Hessian under the Assumption~\ref{ass:iH} are presented in Algorithm~\ref{alg:gnc*}, \ref{alg:gncA*}, respectively. The variants of NCG-A2, NCG-B1, NCG-B2 employing the inexact Hessian can be developed similarly, and hence are omitted.   The convergence guarantee of Algorithm~\ref{alg:gncA*} is given below. 
\begin{theorem}[counterpart of Theorem~\ref{lemma:GNC-A} for inexact Hessian]
\label{lemma:GNC-A-subsampled}
Suppose Assumption~\ref{ass:iH} holds. 
%\frac{\max(\epsilon_1, \epsilon_2)}{12}$
 The iH-NCG-A algorithm terminates at some iteration $j_*$ with
\begin{align}
	j_*\leq 1+\max\left(\frac{24L_2^2}{\epsilon_2^3},\frac{2L_1}{\epsilon_1^2}\right)\Delta,
\end{align}
satisfying $\|\nabla f(\x_{j_*})\|\leq\epsilon_1$, and with probability $1-\delta$, $\lambda_{\text{min}}\left(\nabla^2f(\x_{j_*})\right)\geq -2\max(\epsilon_2,\epsilon_1)$.
Furthermore, the $j$-th iteration requires time at most 
\[
O\left(T_h \frac{\sqrt{M}}{\max(\epsilon_2, \|\nabla f(\x_j)\|)^{1/2}}\log\left(\frac{d}{\delta'}\right)\right)
\]
\end{theorem}
{\bf Remark:} We can see that the iteration complexity and worst-case time complexity of iH-NCG-A is the same to that in Theorem~\ref{lemma:GNC-A} up to a constant factor. Based on iH-NCG-A,  it is straightforward (hence details omitted here) to develop a similar variant of NCG-B1 using the inexact Hessian, which would enjoy the same complexity as NCG-B1 provided Assumption~\ref{ass:iH} is satisfied. 
\begin{algorithm}[t]
	\caption{The iH-NCG Step: $(\x^+, \v)=\text{iH-NCG}(\x, \varepsilon, \delta,\epsilon_2)$}\label{alg:gnc*}
	\begin{algorithmic}[1]
		\STATE \textbf{Input}:  $\x$, $\varepsilon$, $\delta$
		%\STATE $|\mathcal{S}|\geq \frac{2304K_{\text{max}}^2}{\epsilon_2^2}\log\frac{4d}{\delta'}$ %$T_1 \geq 81Lc^2\epsilon_0^{2\theta -1}$
		\STATE Find a vector $\v$ such that $\|\v\|=1$ and, with probability at least $1-\delta$,
		\[
		\lambda_{\min}(H(\x))\geq \v^{\top}H(\x)\v - \varepsilon
		\]
		using the Lanczos method. 
		%\IF{$\frac{\epsilon_2^3}{24L_2^2}>\frac{\|\nabla f(\x)\|^2}{2L_1}$}
		\IF{$-\frac{\epsilon_2^2}{2L_2^2}\v^\top H(\x)\v-\frac{5\epsilon_2^3}{24L_2^2}>\frac{\|\nabla f(\x)\|^2}{2L_1}$}
		\STATE Compute $\x^+  = \x - \frac{\epsilon_2}{L_2}\text{sign}(\v^{\top}\nabla f(\x))\v$
		\ELSE
		\STATE Compute $\x^+ = \x - \frac{1}{L_1}\nabla f(\x)$
		\ENDIF
		\STATE return $\x^+, \v$
	\end{algorithmic}
\end{algorithm}
\begin{algorithm}[t]
	\caption{iH-NCG-A: $(\x_0, \epsilon_1, \epsilon_2, \delta)$}\label{alg:gncA*}
	\begin{algorithmic}[1]
		\STATE \textbf{Input}:  $\x_0$, $\epsilon_1, \epsilon_2$, $\delta$
		\STATE $\x_1=\x_0$, $\delta' = \delta /(1+\max\left(\frac{24L_2^2}{\epsilon_2^3}, \frac{2L_1}{\epsilon_1^2}\right)\Delta)$%, $|\mathcal{S}|\geq \frac{2304K_{\text{max}}^2}{\epsilon_2^2}\log\frac{4d}{\delta'}$ %$T_1 \geq 81Lc^2\epsilon_0^{2\theta -1}$
		\FOR{$j=1,2,\ldots,$}
		\STATE  $(\x_{j+1}, \v_j) = \text{iH-NCG}(\x_j,  \max(\epsilon_2, \|\nabla f(\x_j)\|)/2,  \delta', \epsilon_2)$
		
		%Find a vector $\v_j$ such that $\|\v_j\|=1$ and, with probability at least $1-\delta'$,
		%\[
		%\lambda_{\min}(\nabla^2 f(\x_j))\geq \v_j^{\top}\nabla^2 f(\x_j)\v_j - \max(\epsilon_2, \|\nabla f(\x_j)\|)/2
		%\]
		%using a leading eigen-vector computation. 
		\IF{ $\v_j^{\top}H(\x_j)\v_j> -\epsilon_2/2$ and $\|\nabla f(\x_j)\|\leq \epsilon_1$}
		\RETURN $\x_j$
		%\ELSIF{$\frac{2|\v_j^{\top}\nabla f^2(\x_j)\v_j|}{3L_2^2}>\frac{\|\nabla f(\x_)\|^2}{2L_1}$}
		%\STATE Compute $\x_{j+1}  = \x_j - \frac{2|\v_j^{\top}\nabla^2 f(\x_j)\v_j|}{L_2}sign(\v_j^{\top}\nabla f(\x_j)\v_j)\v_j$
		%\ELSE
		%\STATE Compute $\x_{j+1} = \x_j - \frac{1}{L_1}\nabla f(\x_j)$
		\ENDIF
		%\STATE 
		\ENDFOR
	\end{algorithmic}
\end{algorithm}
\vskip2ex
To prove the above theorem, similar to Lemma~\ref{lem:NCG}, we first present a result of objective decrease for each iH-NCG step. 
\begin{lemma}\label{lem:NCG-iH}
Suppose Assumption~\ref{ass:iH} holds. For any $\varepsilon, \delta\in(0, 1)$, the update $$\x_{j+1}=\text{iH-NCG}(\x_j, \varepsilon,  \delta, \epsilon_2)$$ yields 
\begin{align*}
f(\x_j) - f(\x_{j+1})\geq \max\left\{\left(\frac{-\epsilon_2^2\v_j^\top H(\x_j)\v_j}{2L_2^2} - \frac{5\epsilon_2^3}{24L_2^2}\right), \frac{\|\nabla f(\x_j)\|^2}{2L_1}\right\}
\end{align*}
If $\v_j^\top H(\x_j)\v_j\leq -\epsilon_2/2$, we have 
\[
f(\x_j) - f(\x_{j+1})\geq\max\left(\frac{\|\nabla f(\x_j)\|^2}{2L_1}, \frac{\epsilon_2^3}{24L_2^2}\right)\]
\end{lemma}

\begin{proof}[Proof of Theorem~\ref{lemma:GNC-A-subsampled}]
	 Following the similar analysis in the proof of Theorem~\ref{lemma:GNC-A}, for all iterations $j\leq j_*-1$ we consider three cases: Case 1: $\|\nabla f(\x_j)\|> \epsilon_1$ and $\v_j^{\top}H(\x_j)\v_j\leq  -\epsilon_2/2$, Case 2:  $\|\nabla f(\x_j)\|\leq\epsilon_1$ and $\v_j^{\top} H(\x_j)\v_j\leq  -\epsilon_2/2$ and Case 3:  $\|\nabla f(\x_j)\|> \epsilon_1$ and $\v_j^{\top}H(\x_j)\v_j>  - \epsilon_2/2$. In either case, following Lemma~\ref{lem:NCG-iH}, we have 
\begin{align}
\label{ineqn:sufficientdecrease}
\min\left(\frac{\epsilon_1^2}{2L_1}, \frac{\epsilon_2^3}{24L_2^2}\right)\leq f(\x_j) - f(\x_{j+1}),
\end{align}
from which we can prove the bound on $j_*$ as in the theorem. 

%Now we show the detailed proof of (\ref{ineqn:sufficientdecrease}) by considering three cases when $j\leq j_*-1$.
%Case 1: $\|\nabla f(\x_j)\|> \epsilon_1$ and $\v_j^{\top}H(\x_j)\v_j\leq  -\epsilon_2/2$, according to Lemma~\ref{lem:NCG-iH} we have
%\begin{align*}
%\max\left(\frac{\epsilon_2^3}{24L_2^2}, \frac{\epsilon_1^{2}}{2L_1}\right)\leq f(\x_j) - f(\x_{j+1})
%\end{align*}
%Case 2:  $\|\nabla f(\x_j)\|\leq\epsilon_1$ and $\v_j^{\top} H(\x_j)\v_j\leq  -\epsilon_2/2$, it follows from Lemma~\ref{lem:NCG-iH} that
%\begin{align*}
%\frac{\epsilon_2^3}{24L_2^2}\leq f(\x_j) - f(\x_{j+1})
%\end{align*}
%Case 3:  $\|\nabla f(\x_j)\|> \epsilon_1$ and $\v_j^{\top}H(\x_j)\v_j>  - \epsilon_2/2$. 
%\begin{align*}
%\frac{\epsilon_1^2}{2L_1}\leq f(\x_j) - f(\x_{j+1})
%%\end{align*}
%In any case, we have
%\begin{align*}
%\min\left(\frac{\epsilon_1^2}{2L_1}, \frac{\epsilon_2^3}{24L_2^2}\right)\leq f(\x_j) - f(\x_{j+1})
%\end{align*}

Next, we establish the convergence guarantee of $\x_{j_*}$. Upon termination, we have $\|\nabla f(\x_{j_*})\|\leq \epsilon_1$ and $\v_{j_*}^{\top}H(\x_{j_*})\v_{j_*}\geq -\epsilon_2/2$. Then with probability at least $1-j_*\delta'$ we have
\begin{align}
\lambda_{\min}(H(\x_{j_*}))\geq \v_{j_*}^{\top}H(\x_{j_*})\v_{j_*} - \max(\epsilon_2, \|\nabla f(\x_{j_*})\|)/2 \geq  -\max(\epsilon_1, \epsilon_2)
\end{align}
Let $D=\max(M, L_1)$, then we have $\|DI - H(\x_{j_*})\|_2 = D - \lambda_{\min}(H(\x_{j_*}))$ and $\|DI - \nabla^2 f(\x_{j_*})\|_2 = D -  \lambda_{\min}(\nabla^2 f(\x_{j_*}))$.
By the regularity condition on the inexact Hessian, we have
\begin{align}
    \|DI -\nabla^2 f(\x_{j_*}) \|_2\leq \|DI - H(\x_{j_*}) \|_2 + \|\nabla^2 f(\x_{j_*}) - H(\x_{j_*})\|_2\leq D - \lambda_{\min}(H(\x_{j_*})) + \epsilon_3
\end{align}
As a result $\lambda_{\min}(\nabla^2 f(\x_{j_*})) \geq \lambda_{\min}(H(\x_{j_*})) - \epsilon_3\geq -\max(\epsilon_1, \epsilon_2) - \epsilon_3$. 
%we know that with probability at least $1-j\delta'/2=1-\delta/2$, the algorithm terminates within $j=1+\max\left(\frac{24L_2^2}{\epsilon_2^3}+\frac{2L_1}{\epsilon_1^2}\right)$ steps. So with probability at least $1-j\delta'=1-\delta$,
%\begin{align*}
%	\lambda_{\text{min}}\left(\nabla^2 f(\x_j)\right)\geq -\max(\epsilon_1,\epsilon_2).
%\end{align*}
The running time spent on the $j$-th iteration follows from Lemma~\ref{lemma:lanczos}. 
\end{proof}

\subsection{A Stochastic Variant of NCG-A}
In this subsection, we propose a stochastic variant of NCG-A for solving a finite or infinite sum problem: 
\begin{align}\label{eqn:stocn1}
	\min_{\x\in\R^d}f(\x)=\frac{1}{n}\sum_{i=1}^{n}f_i(\x),
\end{align}
where $n$ could be very large and even infinity and $f_i(\x)$ satisfies the first three bullets in Assumption \ref{ass:1} but not necessarily convex. When $n$ is infinity, the problem becomes an instance of stochastic optimization:
\begin{align}\label{eqn:stocn}
	\min_{\x\in\R^d}f(\x)=\E_{\xi}[f(\x; \xi)],
\end{align}
where $f(\x; \xi)$ is a random function satisfying  the first three bullets in Assumption \ref{ass:1} but not necessarily convex. For a stochastic non-convex optimization problem~(\ref{eqn:stocn}), we abuse the notation  $f_i(\x)=f(\x; \xi_i)$ for a random variable $\xi_i$. 
To develop the stochastic variant, we consider to use sub-samples to compute a stochastic gradient and a stochastic Hessian. Denote by $\mathcal S_1$ and $\mathcal S_2$ two random sets with elements sampled uniformly at random from $\{1, \ldots, n\}$ for~\eqref{eqn:stocn1} (or sampled from the involved distribution for~\eqref{eqn:stocn}).  Then we compute a stochastic (sub-sampled) gradient $\g(\x)$ and a stochastic (sub-sampled) Hessian $H(\x)$ as follows: 
\begin{align*}
\mathbf g(\x)  = \frac{1}{|\mathcal S_1|}\sum_{i\in\mathcal S_1}\nabla f_i(\x),\quad H(\x)  = \frac{1}{|\mathcal S_2|}\sum_{i\in\mathcal S_2} \nabla^2 f_i(\x)
\end{align*}
To derive the complexity result, we make the following assumption: 
\begin{ass}\label{ass:stoc}
    For any $\delta'\in(0,1)$  and $\x\in\R^d$, assume that there exist $\epsilon_3\leq \epsilon_2/24$ and $\epsilon_4\leq \min(\frac{1}{2\sqrt{2}}\epsilon_1, \epsilon_2^2/(24L_2))$ such that  with probability $1-\delta'$, 
    \begin{align*}
        \|H(\x)-\nabla^2 f(\x)\|_2\leq\epsilon_3, \quad \|\g(\x) - \nabla f(\x)\|\leq \epsilon_4
    \end{align*}
\end{ass}
\begin{algorithm}[t]
	\caption{The stochastic NCG Step: $(\x^+, \v)=\text{NCG-S}(\x, \varepsilon, \delta,\epsilon_2)$}\label{alg:sgnc}
	\begin{algorithmic}[1]
		\STATE \textbf{Input}:  $\x$, $\varepsilon$, $\delta$
		\STATE Compute a sub-sampled gradient $\g(\x)$ and a sub-sampled Hessian $H(\x)$
		%\STATE $|\mathcal{S}|\geq \frac{2304K_{\text{max}}^2}{\epsilon_2^2}\log\frac{4d}{\delta'}$ %$T_1 \geq 81Lc^2\epsilon_0^{2\theta -1}$
		\STATE Find a vector $\v$ such that $\|\v\|=1$ and, with probability at least $1-\delta$,
		\[
		\lambda_{\min}(H(\x))\geq \v^{\top}H(\x)\v - \varepsilon
		\]
		using a randomized algorithm $\mathcal A$ as in Lemma~\ref{lem:approxPCA}. 
		%\IF{$\frac{\epsilon_2^3}{24L_2^2}>\frac{\|\nabla f(\x)\|^2}{2L_1}$}
		\IF{$-\frac{\epsilon_2^2}{2L_2^2}\v^\top H(\x)\v-\frac{11\epsilon_2^3}{48L_2^2}>\frac{\|\g(\x)\|^2}{4L_1} -  \frac{\epsilon_1^2}{8L_1}$}
		\STATE Compute $\x^+  = \x - \frac{\epsilon_2}{L_2}\text{sign}(\v^{\top}\g(\x))\v$
		\ELSE
		\STATE Compute $\x^+ = \x - \frac{1}{L_1}\g(\x)$
		\ENDIF
		\STATE return $\x^+, \v$
	\end{algorithmic}
\end{algorithm}
\begin{algorithm}[t]
	\caption{SNCG: $(\x_0, \epsilon_1, \alpha, \delta)$}\label{alg:sgncA}
	\begin{algorithmic}[1]
		\STATE \textbf{Input}:  $\x_0$, $\epsilon_1, \alpha$, $\delta$
		\STATE $\x_1=\x_0$, $\epsilon_2 = \epsilon_1^\alpha$, $\delta' = \delta /(1+\max\left(\frac{48L_2^2}{\epsilon_2^3}, \frac{8L_1}{\epsilon_1^2}\right)\Delta)$%, $|\mathcal{S}|\geq \frac{2304K_{\text{max}}^2}{\epsilon_2^2}\log\frac{4d}{\delta'}$ %$T_1 \geq 81Lc^2\epsilon_0^{2\theta -1}$
		\FOR{$j=1,2,\ldots,$}
		\STATE let $\g(\x_j)$ be a sub-sampled gradient
		\STATE  $(\x_{j+1}, \v_j) = \text{NCG-S}(\x_j,  \max(\epsilon_2, \|\g(\x_j)\|^\alpha)/2,  \delta', \epsilon_2)$
		
		%Find a vector $\v_j$ such that $\|\v_j\|=1$ and, with probability at least $1-\delta'$,
		%\[
		%\lambda_{\min}(\nabla^2 f(\x_j))\geq \v_j^{\top}\nabla^2 f(\x_j)\v_j - \max(\epsilon_2, \|\nabla f(\x_j)\|)/2
		%\]
		%using a leading eigen-vector computation. 
		\IF{ $\v_j^{\top}H(\x_j)\v_j> -\epsilon_2/2$ and $\|\g(\x_j)\|\leq \epsilon_1$}
		\RETURN $\x_j$
		%\ELSIF{$\frac{2|\v_j^{\top}\nabla f^2(\x_j)\v_j|}{3L_2^2}>\frac{\|\nabla f(\x_)\|^2}{2L_1}$}
		%\STATE Compute $\x_{j+1}  = \x_j - \frac{2|\v_j^{\top}\nabla^2 f(\x_j)\v_j|}{L_2}sign(\v_j^{\top}\nabla f(\x_j)\v_j)\v_j$
		%\ELSE
		%\STATE Compute $\x_{j+1} = \x_j - \frac{1}{L_1}\nabla f(\x_j)$
		\ENDIF
		%\STATE 
		\ENDFOR
	\end{algorithmic}
\end{algorithm}
To better leverage the finite-sum structure of $H(\x)$, we will use an approximate PCA algorithm for computing its minimum eigen-value. To this end, we first present a result  for computing the minimum eigen-value of $H(\x)$. 
\begin{lemma}\label{lem:approxPCA}
Let $H = \frac{1}{m}\sum_{i=1}^mH_i$ where $\|H_i\|_2\leq L_1$. There exists a randomized algorithm $\mathcal A$ such that with probability at least $1- \delta$, $\mathcal A$ produces a unit vector $\v$ satisfying  $\lambda_{\min}(H)\geq \v^{\top}H\v - \varepsilon$ with a time complexity of $\widetilde O(T_h^1\max\{m, m^{3/4}\sqrt{L_1/\varepsilon}\})$, where $T_h^1$ denotes the time of computing $H_i\v$ and $\widetilde O$ suppresses a logarithmic term in $\delta, d, 1/\varepsilon$. 
\end{lemma}
{\bf Remark:} The randomized algorithms proposed in~\citep{DBLP:conf/nips/ZhuL16,DBLP:conf/icml/GarberHJKMNS16} can serve this purpose.  

Next, we present a stochastic variant of NCG-A2 in Algorithm~\ref{alg:sgncA} assuming that $\epsilon_2 = \epsilon_1^{\alpha}$ for $\alpha\in(0,1]$ and the Assumption~\ref{ass:stoc} holds, which employs the stochastic NCG step presented in Algorithm~\ref{alg:sgnc}.  Denote by $T^1_g$ the time complexity for computing one stochastic gradient $\nabla f_i(\x)$ and by $T^1_h$ the time complexity for computing one $\nabla^2 f_i(\x)\v$ Hessian-vector product. The convergence guarantee of SNCG for finding an $(\epsilon_1, \epsilon_2)$-second-order stationary point is presented below. 
\begin{theorem}
\label{thm:GNC-A-stoc}
Suppose Assumption~\ref{ass:stoc} holds and $\epsilon_2 = \epsilon_1^\alpha$ for $\alpha\in(0,1]$.
%\frac{\max(\epsilon_1, \epsilon_2)}{12}$. 
With probability $1-\delta$, the SNCG algorithm terminates at some iteration $j_*$ with
\begin{align}
	j_*\leq 1+\max\left(\frac{48L_2^2}{\epsilon_2^3},\frac{8L_1}{\epsilon_1^2}\right)\Delta,
\end{align}
and upon termination it holds that $\|\nabla f(\x_{j_*})\|\leq2\epsilon_1 $ and $\lambda_{\text{min}}\left(\nabla^2f(\x_{j_*})\right)\geq -2\epsilon_2$ with probability $1-3\delta$.
Furthermore, the $j$-th iteration requires time at most 
\[
\widetilde O\left(T_h^1|\mathcal S_2| + T_h^1|\mathcal S_2|^{3/4} \frac{\sqrt{L_1}}{\max(\epsilon_2, \|\g(\x_j)\|^\alpha)^{1/2}} + |\mathcal S_1|T_g^1\right)
\]
\end{theorem}
To fully quantify the time complexity of SNCG, we need to unveil the order of $|\mathcal S_1|$ and $|\mathcal S_2|$ to ensure Assumption~\ref{ass:stoc} holds. To this end, we explore the concentration inequalities for sub-sampled Hessian and sub-sampled gradient. 
\begin{lemma}
	\label{lem:Hc}
	For any $\epsilon_3,\delta'\in(0,1)$, $\x\in\R^d$, when $|\mathcal{S}_2|\geq\frac{16L_1^2}{\epsilon_3^2}\log(\frac{2d}{\delta'})$, we have $\|H(\x)\|_2\leq L_1$ and
	\begin{align*}
		\Pr(\|H(\x)-\nabla^2 f(\x)\|_2\leq\epsilon_3)\geq 1-\delta'.
	\end{align*}
\end{lemma}
The above lemma can be proved by using matrix concentration inequalities. Please see \citep{peng16inexacthessian}[Lemma 4] for a proof.  

\begin{lemma}
	\label{lem:gc}
	Assume that $\E[\exp(\|\nabla f_i(\x) - \nabla f(\x)\|/G)]\leq \exp(1)$ holds for any $\x\in\R^d$. For any $\epsilon_4,\delta\in(0,1)$, $\x\in\R^d$, when %$|\mathcal{S}_1|\geq\frac{16G^2}{\epsilon_4^2}\log(\frac{2}{\delta})$, 
	$|\mathcal{S}_1|\geq\frac{4G^2(1+3\log^2(1/\delta')}{\epsilon_4^2})$,
	we have 
	\begin{align*}
		\Pr(\|\g(\x)-\nabla  f(\x)\|\leq\epsilon_4)\geq 1-\delta'.
	\end{align*}
\end{lemma}
{\bf Remark:} The above lemma makes an extra assumption of $\E[\exp(\|\nabla f_i(\x) - \nabla f(\x)\|/G)]\leq \exp(1)$ for any $\x\in\R^d$,  which is necessary for establishing a high probability result. When each component function $f_i(\x)$ is $G$-Lipschitz continuous, this assumption is automatically satisfied. The lemma can be proved by using large deviation theorem of vector-valued martingales (e.g., see~\citep{Ghadimi:2016:MSA:2874819.2874863}[Lemma 4]). 

Based on the above two lemmas, we can derive a corollary of Theorem~\ref{thm:GNC-A-stoc} to fully quantify the time complexity of SNCG.  
\begin{cor}
\label{cor:GNC-A-stoc}
Assume that $\E[\exp(\|\nabla f_i(\x) - \nabla f(\x)\|/G)]\leq \exp(1)$ holds for any $\x\in\R^d$, and $\epsilon_2 = \epsilon_1^\alpha$ for $\alpha\in(0,1]$. Set $|\mathcal{S}_1|=\max\left(\frac{32G^2}{\epsilon_1^2},\frac{2304G^2L_2^4}{\epsilon_2^4}\right)(1+3\log^2(\frac{2}{\delta'}))$ and $|\mathcal{S}_2|=\frac{9216L_1^2}{\epsilon_2^2}\log(\frac{4d}{\delta'})$.
%\frac{\max(\epsilon_1, \epsilon_2)}{12}$. 
With probability $1-\delta$, the SNCG algorithm terminates at some iteration $j_*$ with
\begin{align}
	j_*\leq 1+\max\left(\frac{48L_2^2}{\epsilon_2^3},\frac{8L_1}{\epsilon_1^2}\right)\Delta,
\end{align}
and upon termination it holds that $\|\nabla f(\x_{j_*})\|\leq2\epsilon_1 $ and $\lambda_{\text{min}}\left(\nabla^2f(\x_{j_*})\right)\geq -2\epsilon_2$ with probability $1-3\delta$.
Furthermore, the worst-case time complexity of SNCG is given by 
\[
\widetilde O\left(T_h^1 \max\left(\frac{1}{\epsilon_1^2\epsilon_2^{2}}, \frac{1}{\epsilon_2^{5}}\right) +T_g^1\max\left(\frac{1}{\epsilon_1^4},   \frac{1}{\epsilon_1^2\epsilon_2^3}, \frac{1}{\epsilon_2^7}, \frac{1}{\epsilon_1^2\epsilon_2^4}\right)\right)
\]
\end{cor}

{\bf Remark:} If  $T^h_1$ is linearly dependent on the dimensionality, to the best of our knowledge, the above result is the first high-probability second-order convergence result for stochastic non-convex optimization with a time complexity independent of the sample size and almost linear in the problem's dimsionality.   Fix $\epsilon_2=\sqrt{\epsilon_1}$, we can  compare with a recent algorithm called Natasha2 proposed in~\citep{natasha2}, which  uses a time complexity of $\widetilde O\left( T^1_g\epsilon_1^{-3.5}+T^1_h\epsilon_1^{-2.5} \right)$ to find  an $(\epsilon_1, \sqrt{\epsilon_1})$-second-order stationary solution with only a constant probability $2/3$. In contrast, SNCG achieves this with a time complexity of $\widetilde O(T^1_g\epsilon_1^{-4} + T^1_h\epsilon_1^{-3})$ in high probability. Although  Natasha2's time complexity has a better dependence on $\epsilon_1$, SNCG is a stochastic algorithm  that comes with a {\bf high probability} second-order convergence result and has a time complexity independent of the sample size and almost linear in the problem's dimensionality. It remains an open question whether the result in~\citep{natasha2} can be boosted to a high probability result.

\subsection{Some Detailed Comparison}
In this subsection, we provide more details on the comparison with two recent works~\citep{clement17,peng16inexacthessian}, which is worth more discussions.  We will compare with these  works from several aspects: (i) algorithm design;  (ii) per-iteration costs;  (iii) convergence guarantee and worst-case time complexity; 
\begin{itemize}
    \item 
Algorithm Design. The algorithms in \citep{clement17} are based on a line-search scheme. A search direction is first determined among the negative gradient, the negative curvature, the Newton step and the regularized Newton step, and then a step size is found by bactracking line search. The Newton step or the regularized Newton step are solved by the conjugate gradient method. It is unclear whether their methods can be modified to incorporate inexact Hessian. In addition, the conjugate gradient method might fail due to the noise estimation of the smallest eigen-value of the Hessian.   \cite{peng16inexacthessian} focused on developing variants of the trust-region methods~\citep{conn2000trust} and adaptive cubic-regularization method~\citep{Cartis:2012:CBS:2076038.2076303} to incorporate inexact Hessian. The sub-problems (the trust-region problem or the cubic regularized problem) in their methods are solved approximately to satisfy certain conditions such that the search direction are at least as good as the negative gradient direction and the negative curvature directions. Hence, they proposed to solve the sub-problems over a two-dimensional space spanned by the negative gradient and the negative curvature direction.  In comparison, the updating steps in this paper are much simpler than that in~\citep{clement17,peng16inexacthessian}. On the other hand, the algorithms in \citep{clement17,peng16inexacthessian} do not require the Lipschitz constants $L_1, L_2$ for updating the solution, and they are adaptive to the local Lipschitz constants. Although our algorithms are presented with the knowledge of $L_1$ and $L_2$, however,  the backtracking search technique can be easily incorporated into our NCG step as used in previous algorithms~\citep{Nesterov_Composite,doi:10.1137/16M1087801} to search for local Lipschitz constants. 
\item Per-iteration costs. The major cost of the NCG step lie at computing a noisy negative curvature. Besides this cost, the algorithms in~\citep{clement17} may need to compute the Newton or regularized Newton steps and the algorithms in~\citep{peng16inexacthessian} need to solve the trust-region or cubic regularized problems. If we are only concerned with time complexity of each iteration, each step of \citep{clement17}'s inexact algorithm is $\widetilde O(1/\sqrt{\epsilon_2})$. For~\citep{peng16inexacthessian}, let us consider the variants that solve the sub-problems over a two-dimensional space spanned by the negative gradient and the negative curvature directions. So the major costs lie at computing the negative curvature direction, which is to compute an approximate negative curvature $\s_E$ such that ${\s_E}^{\top} H(\x)\s_E\leq -\nu \epsilon_2\|\s_E\|^2$  for $\lambda_{\min}(H(\x))\leq \epsilon_2$ and some $\nu$ according to the result in~\citep{peng16inexacthessian}~\footnote{See their Appendix B. }. If the Lanczos method is applied to  the matrix $\widehat H =L_1I - H(\x)$ for finding $\s_E$ such that 
\begin{align}\label{eqn:sE}
\frac{\s_E^{\top}\widehat H\s_E}{\|\s_E\|^2}\geq (1 - \frac{\epsilon_2}{2L_1 + \epsilon_2})\lambda_{\max}(\widehat H), 
\end{align}
from which we can derive that 
\begin{align*}
\s_E^{\top}H\s_E\leq \frac{\epsilon_2}{2L_1+\epsilon_2}L_1\|\s_E\|^2 + \frac{2L_1}{2L_1+\epsilon_2}\lambda_{\min}(H)\|\s_E\|^2\leq - \frac{L_1}{L_1 + \epsilon_2}\epsilon_2\|\s_E\|^2
\end{align*}
The time complexity of the Lanczos method for achieving~(\ref{eqn:sE}) is $\widetilde O(1/\sqrt{\epsilon_2})$. Hence the time complexity of each step in \citep{peng16inexacthessian}'s algorithms is at least $\widetilde O(1/\sqrt{\epsilon_2})$. In contrast, the time complexity of each step in our algorihtms is $\widetilde O(1/\sqrt{\max(\epsilon_2, \|\nabla f(\x_j)\|^\alpha)})$, where $\alpha\leq 1$ is some constant, which could be much less than $\widetilde O(1/\sqrt{\epsilon_2})$.  

\item Convergence Guarantee and Worse-case time complexity. The algorithms in~\citep{clement17} are guaranteed to find two solutions $\x_k, \x_{k+1}$ to satisfy 
\begin{align}\label{eqn:cleg}
\min(\|\nabla f(\x_k)\|, \|\nabla f(\x_{k+1})\| )\leq \epsilon_1, \quad \lambda_{\min}(\nabla^2 f(\x_k))\geq - \epsilon_2, 
\end{align}
which is slightly different from~(\ref{eqn:sog}) that is on a single point. Note that the condition~(\ref{eqn:sog}) implies~(\ref{eqn:cleg}). The analysis provided in~\citep{peng16inexacthessian} is to guarantee that the found solution $\x$ satisfies 
\begin{align}
    \|\nabla f(\x)\|\leq \epsilon_1, \quad \lambda_{\min}(H(\x))\geq - \epsilon_2, 
\end{align}
where $H(\x)$ is an inexact Hessian, which is weaker than~(\ref{eqn:sog}). Nevertheless, the same result~(\ref{eqn:sog}) can be derived under a stronger condition on the inexact Hessian such as~(\ref{eqn:ine})  than that provided in~\citep{peng16inexacthessian} as mentioned at the beginning of this section. In the following comparisons, we assume the same result~(\ref{eqn:sog}) for \citep{peng16inexacthessian}. 

The worse-case time complexity of the inexact variant of the algorithm in~\citep{clement17} is given by 
\begin{align}
\label{complexity:sw}
\widetilde{O}\left(\frac{1}{\sqrt{\epsilon_2}}\left(\frac{\epsilon_2^3}{\epsilon_1^3}+\frac{1}{\epsilon_1^{3/2}}+\frac{1}{\epsilon_2^3}+\frac{(\sqrt{\epsilon_1+\epsilon_2^2}+\epsilon_2)^3}{\epsilon_1^3}\right)T_h\right).
\end{align}
It is not difficult show that when $\epsilon_2\leq \epsilon_1^{1/2}$, the above time complexity is dominated $\widetilde O\left(\frac{T_h}{\epsilon_2^{7/2}}\right)$. In comparison, the worse-case time complexity of NCG-B is also dominated by $\widetilde O\left(\frac{T_h}{\epsilon_2^{7/2}}\right)$ by simply considering each iteration takes $\widetilde O(T_h/\sqrt{\epsilon_2})$ time. 

For~\citep{peng16inexacthessian}, we consider two algorithms (their Algorithm 1 and Algorithm 2) that solve the sub-problems approximately over a two dimensional subspace. The iteration complexity of their Algorithm 1 is $O(\max(\epsilon_1^{-2}\epsilon_2^{-1}, \epsilon_2^{-3}))$, which gives the worse-case time complexity of $\widetilde O\left(\max\left(\frac{T_h}{\epsilon_1^{2}\epsilon_2^{3/2}}, \frac{T_h}{\epsilon_2^{7/2}}\right)\right)$ by considering the time complexity of each iteration discussed above, which is dominated by  $\widetilde O\left(\frac{T_h}{\epsilon_1^{2}\epsilon_2^{3/2}}\right)$ and hence is no better than $\widetilde O\left(\frac{T_h}{\epsilon_2^{7/2}}\right)$ for $\epsilon_2\leq \epsilon_1^{1/2}$. Their Algorithm 2 has an iteration complexity of $O(\max(\epsilon_1^{-2}, \epsilon_2^{-3})$, yielding a worst-case time complexity of  $\widetilde O\left(\max\left(\frac{T_h}{\epsilon_1^{2}\epsilon_2^{1/2}}, \frac{T_h}{\epsilon_2^{7/2}}\right)\right)$. This is worse than $\widetilde O\left(\frac{T_h}{\epsilon_2^{7/2}}\right)$ when $\epsilon_2\geq \epsilon_1^{2/3}$. It is notable that they also provide a result with optimal iteration complexity of $O(\max(\epsilon_1^{-3/2}, \epsilon_2^{-3}))$. However, such a result requires the algorithm to solve the sub-problems over progressively embedded subspaces with increasingly higher dimensions to satisfy a certain criterion, which leaves time complexity unclear. 

\end{itemize}

\subsection{Application to Strict-Saddle functions}
%Our developed schemes (GNC-A2', GNC-A3) have two significant advantages. The first is that it is data-adaptive, which would yield better practical performance. The second is that it allows arbitrary accuracy level of both the first and second order guarantee, which is attractive for some strict saddle functions.
In this subsection, we consider several examples of strict-saddle functions and demonstrate how the proposed algorithms can be used to find a nearly local optimal solution. We use the following definition of strict-saddle function from~\citep{jin2017escape}.

\begin{ass}
	\label{ass:strict-saddle1}
	A function $f(\x)$ is $(\theta, \gamma, \zeta)$-strict saddle. That is, for any $\x$, at least one of the following holds: 
	\begin{itemize}
		\item $|\nabla f(\x)|\geq \theta$
		\item $\lambda_{\min}(\nabla^2 f(\x))\leq -\gamma$
		\item $\x$ is $\zeta$-close to $\mathcal X_*$ - the set of local minima. 
	\end{itemize}
\end{ass}
For a strict saddle function satisfying the Assumption \ref{ass:strict-saddle1}, to guarantee the solution is $\zeta$-close to the set of local minima, we can  solve the optimization problem to find an $(\theta, \gamma)$-second-order stationary point $\x$ such that:
\begin{align}
\label{guarantee:us}
	\|\nabla f(\x)\|\leq\theta,\;\;\lambda_{\text{min}}\left(\nabla^2 f(\x)\right)\geq-\gamma.
\end{align}
Below, we adopt several examples of strict saddle functions studied in previous work and show the time complexities of our algorithms (NCG-A, NCG-B) for finding a solution that is $\zeta$-close to a local minima.
\paragraph{Matrix Factorization} Consider the the following symmetric low-rank matrix factorization problem:
\begin{align*}
	\min_{\U\in\R^{d\times r}}f(\U)=\frac{1}{2}\|\U\U^\top-\mathbf{M}\|_F^2,
\end{align*}
where $\mathbf{M}\in\R^{d\times d}, \text{rank}(\mathbf{M})=r$. Denote $\sigma_i(\mathbf{M})$ by the $i$-th largest singular value of $\mathbf{M}$. As is shown in Lemma 6 and Lemma 7 of~\citep{jin2017escape}, the function is $(\theta = \frac{1}{24}\sigma_r^{3/2}, \gamma=1/3\sigma_r, \zeta=1/3\sigma_r^{1/2})$-strict-saddle function, and have Lipschitz continuous gradient and Hessian with $L_1 = 8\mathcal T$ and $L_2 = 12\sqrt{\mathcal T}$ for $\|U\|_2^2<\mathcal T$. To find a solution that is $O(\sigma_r^{1/2})$-close to a local minima, we can set $\epsilon_1 = \frac{1}{24}\sigma_r^{3/2}$ and $\epsilon_2 =\left(\frac{1}{24}\right)^{2/3}\sigma_r\leq\frac{1}{3}\sigma_r$, and run the proposed algorithms (NCG-A, NCG-B).  Their worse-case time complexity is in the same order of $\widetilde{O}(T_h\sigma_r^{-7/2})$. This result has a better dependence on $\sigma_r$ than that established in~\citep{jin2017escape}, which is $\widetilde O(T_g\sigma_r^{-4})$. 

\paragraph{Dictionary Learning over the Sphere}This problem considered in~\citep{sun2015complete} is to recover an invertible matrix $A\in\R^{n\times n}$ (dictionary) from response $Y\in\R^{n\times p}$, where $Y=AX$, and $X\in\R^{n\times p}$ (coefficient matrix) is sufficiently sparse. In addition, assume $X$ follows the Bernoulli-Gaussian model with rate $\beta$: $X_{ij}=\Omega_{ij}V_{ij}$, with $\Omega_{ij}\sim\text{Bernoulli}(\beta)$ and $V_{ij}\sim N(0,1)$. Under this setting, the row spaces of $X$ and $Y$ are identical. If we want to find the direction with highest sparsity level in the row space of $Y$, the following optimization problem is formulated:
\begin{align*}
	\min_{\d\in\R^n}\|\d^\top Y\|_0,\;\text{ s.t. } \d\neq 0.
\end{align*}
According to~\citep{sun2015complete}, this optimization problem can be relaxed to
\begin{align*}
	\min_{\d}f(\d):=\frac{1}{p}\sum_{k=1}^{p}h_\mu(\d^\top\widehat{\y}_k), \text{ s.t. } \|d\|_2=1,
\end{align*}
where $\widehat{Y}$ is a proxy for $Y$, $\widehat{\y}_k$ is the $k$-th column of $\widehat{Y}$, $h_\mu$ is a smooth approximation of the absolute function. One choice of $h_\mu$ is 
\begin{align*}
	h_\mu(z)=\mu\log\left(\frac{\exp(z/\mu)+\exp(-z/\mu)}{2}\right)=\mu\log\cosh(z/\mu),
\end{align*}
which is infinitely differentiable and $\mu$ controls the smoothness level. In Theorem 2.3 and Corollary 2.4 of~\citep{sun2015complete}, it has been shown that when $0<\theta<1/2$, with high probability, the function is $(\theta=c\beta,\gamma=c\beta,\zeta=\sqrt{2}\mu/7)$-strict saddle, where $c$ is a constant and $\beta$ characterize the sparsity level of $X$. To solve this problem, we can set $\epsilon_1 = \epsilon_2= c\beta$ in our algorithms (NCG-A, NCG-B), which have the same worst-case time complexity of  $\widetilde{O}(1/\beta^{7/2})$.

\paragraph{Matrix Sensing} The following  matrix sensing problem is studied in~\citep{ge2017no}:
\begin{align*}
	\min_{\U,\mathbf{V}}\frac{1}{2m}\sum_{i=1}^{m}\left(\text{tr}(\U\mathbf{V}^\top \mathbf{A}_i^\top)-b_i\right)^2+\frac{1}{8}\|\U^\top\U-\mathbf{V}\mathbf{V}^\top\|_F^2,
\end{align*}
where $\U\in\R^{d_1\times r}$, $\mathbf{V}\in\R^{d_2\times r}$, $\mathbf{A}_1,\ldots,\mathbf{A}_m\in\R^{d_1\times d_2}$ are sensing matrices, and $\b_i\in\R$ can be observed. It is shown in Theorem 3 of~\citep{ge2017no} that if for every matrix 
$M$ of rank at most $2r$, $(1-\delta)\|\mathbf{M}\|_F^2\leq\frac{1}{m}\sum_{i=1}^{m}\text{tr}^2(\mathbf{A}_i\mathbf{M}^\top)\leq (1+\delta)\|\mathbf{M}\|_F^2$ holds with $\delta=1/20$, then the objective function is $(\theta,\frac{1}{5}\sigma_r,\frac{20\theta}{\sigma_r})$-strict saddle, where $\sigma_r$ stands for the $r$-th largest singular value of $\mathbf{M}$. Without loss of generality, assume $0<\sigma_r<1$. It is notable that $\theta$ is a free parameter. The larger the $\theta$, the smaller the $\gamma$. Let us set $\epsilon_1=(\sigma_r/5)^{1/\alpha}$, $\epsilon_2=\frac{1}{5}\sigma_r$ in our algorithms,  where $\alpha< 1$. We consider the following several cases of $\alpha$.%and then achieving a solution with the guarantee (\ref{guarantee:us}) implies that solution is $\zeta$-close to the set of local minima. Now we consider the complexity of algorithms when different values of $\alpha$ are used.
\begin{itemize}
\item When $2/3\leq \alpha< 1$, NCG-A, NCG-B  have the same worst-case time complexity of $\widetilde{O}(1/\sigma_r^{7/2})$ for finding a solution that is $O(\sigma_r^{1/\alpha -1})$-close to a local minima.
%\item When $4/7\leq \alpha\leq 2/3$, NCG-B and NCG-C  have the best worst-case time complexity of $\widetilde{O}(1/\sigma_r^{7/2})$ for finding a solution that is $O(\sigma_r^{1/\alpha -1})$-close to a local minima.
\item When $1/2\leq \alpha\leq 2/3$, NCG-B has the best worst-case time complexity of $\widetilde{O}(1/\sigma_r^{7/2})$ for finding a solution that is $O(\sigma_r^{1/\alpha -1})$-close to a local minima.
\item When $0< \alpha\leq 1/2$, NCG-B has the best worst-case time complexity of $\widetilde{O}(1/\sigma_r^{1/\alpha + 3/2})$ for finding a solution that is $O(\sigma_r^{1/\alpha -1})$-close to a local minima.
\end{itemize}
Since the larger the $\alpha$, the closer the solution is to a local optimum. As a result, NCG-B  always has the best worst-case complexity of  $\widetilde{O}(1/\sigma_r^{1/\alpha + 3/2})$ to find a solution that is $O(\sigma^{1/\alpha -1})$-close to a local minima by setting $\alpha\leq 1/2$.

\section{Conclusions}\label{sec:conc}
In this paper, we have developed a new updating step employing a noisy negative curvature direction for non-convex optimization. The novelty of the proposed algorithms lie at that the noise level in approximating the negative curvature is adaptive to the magnitude of the current gradient instead of a prescribed small noise level, which could dramatically reduce the number Hessian-vector products.  Building on the basic step, we have developed several algorithms and established their iteration complexities and worst-case time complexities. We have also developed their variants for using inexact Hessian. Under a mild condition on the inexact Hessian, we obtain the same complexity results as their counterparts with exact Hessian.  We also develop a stochastic algorithm for a finite-sum non-convex optimization problem, which converges to a second-order stationary point in {high probability} with a time complexity independent of the sample size and almost linear in the problem's dimensionality.

\bibliography{ref,all}
\bibliographystyle{plain}
%\section{\appendix} 
\newpage
\appendix
%\section{Proof of Lemmas and Theorems}
\section{\textbf{Proof of Lemma \ref{lem:NCG}}}
\begin{proof}
%According to the update of NCG step, we know that the algorithm either does the gradient descent or the negative curvature descent according which step decreases more objective value. 
Denote $\eta_j=\frac{2|\v_j^\top \nabla^2f(\x_j)\v_j|}{L_2}\text{sign}(\v_j^\top\nabla f(\x_j))$. Let $\x_{j+1}^1 = \x_j - \eta_j \v_j$ denote the updated solution if following $\v_j$ and $\x_{j+1}^2 = \x_j - \nabla f(\x_j)/L_1$ denote the updated solution if following $\nabla f(\x_j)$.  Since $f(\x)$ has a $L_2$-Lipschitz continuous Hessian, we have
\begin{align*}
|f(\x^1_{j+1})-f(\x_{j})+\eta_j\v_j^\top\nabla f(\x_j)-\frac{1}{2}\eta_j^2\v_j^\top\nabla^2f(\x_j)\v_j|\leq\frac{L_2}{6}\|\eta_j\v_j\|^3.
\end{align*}
By noting that $\eta_j\v_j^\top\nabla f(\x_j)\geq 0$, we have
\begin{align*}
    f(\x_j)- f(\x^1_{j+1})\geq -\frac{1}{2}\eta_j^2\v_j^\top\nabla^2f(\x_j)\v_j - \frac{L_2}{6}\|\eta_j\v_j\|^3= \frac{2|\v_j\nabla^2f(\v_j)\v_j|^3}{3L_2^2} \triangleq \Delta_1. 
\end{align*}
By the smoothness of $f(\x)$,  we have
\begin{align*}
f(\x^2_{j+1})& \leq f(\x_j) + \nabla f(\x_j)^{\top}(\x_{j+1} - \x_j) + \frac{L_1}{2}\|\x_{t+1} - \x_t\|^2\\
&= f(\x_j) - \frac{1}{L_1}\|\nabla f(\x_j)\|_2^2 + \frac{L_1\eta^2}{2}\|\nabla f(\x_j)\|^2\\
&\leq f(\x_j) - \frac{1}{2L_1}\|\nabla f(\x_j)\|^2
\end{align*}
As a result,  $f(\x_{j})-f(\x_{j+1}^2)\geq  \frac{\|\nabla f(\x_j)\|^2}{2L_1}\triangleq \Delta_2$.

According to the update rule in NCG,  if $\Delta_1>\Delta_2$, we have $\x_{j+1} = \x_{j+1}^1$ and $f(\x_j) - f(\x_{j+1})\geq \Delta_1 = \max(\Delta_1, \Delta_2)$. If $\Delta_2\geq \Delta_1$, then $\x_{j+1} = \x^2_{j+1}$ and $f(\x_j) - f(\x_{j+1})\geq \Delta_2 = \max(\Delta_1, \Delta_2)$. In both cases, we have $f(\x_j) - f(\x_{j+1})\geq \max(\Delta_1, \Delta_2)$. 
\end{proof}

\section{\textbf{Proof of Theorem \ref{thm:gnc-a2:iteration}}}

\begin{proof}
Let $j_*$ denote the $j$ such that the algorithm terminates. Then for all $j<j_*$, we have $\|\nabla f(\x_j)\|> \epsilon_1$, or $\v_j^{\top}\nabla^2 f(\x_j)\v_j\leq  - \epsilon_2/2$. 
\begin{comment}Let $\x_{j+1}^1$ be the updated solution according to the NCD update and $\x_{j+1}^2$ be the updated solution according to the GD update. We can prove that
\begin{align}
f(\x_{j+1}^1) - f(\x_j)\leq - \frac{2|\v_j^{\top}\nabla ^2 f(\x_j)\v_j|^3}{3L_2^2} = - \Delta_j^1\label{eqn:decrease1}\\
f(\x_{j+1}^2) - f(\x_j) \leq - \frac{1}{2L_1}\|\nabla f(\x_j)\|^2 = - \Delta_j^2
\end{align}
where the inequality (\ref{eqn:decrease1}) follows from the $L_2$-Lipschitz continuity of the Hessian.
If $\Delta_j^1>\Delta_j^2$, then $f(\x_{j+1}) - f(\x_j) =  f(\x_{j+1}^1) - f(\x_j)\leq -\Delta_j^1$. If $\Delta_j^2>\Delta^j_1$, then  $f(\x_{j+1}) - f(\x_j) =  f(\x_{j+1}^2) - f(\x_j)\leq -\Delta_j^2$. In both cases, we have
\begin{align*}
\max(\Delta_j^1, \Delta_j^2)\leq f(\x_j ) - f(\x_{j+1})
\end{align*}
\end{comment}
According to Lemma~\ref{lem:NCG}, we have
\begin{align*}
f(\x_j) - f(\x_{j+1})\geq \max\left( \frac{2|\v_j^{\top}\nabla ^2 f(\x_j)\v_j|^3}{3L_2^2}, \frac{\|\nabla f(\x_j)\|^{2}}{2L_1}\right)
\end{align*}
Let us consider three cases. Case 1: $\|\nabla f(\x_j)\|> \epsilon_1$ and $\v_j^{\top}\nabla^2 f(\x_j)\v_j\leq  -\epsilon_2/2$, then we have
\begin{align*}
\max\left(\frac{\epsilon_2^3}{12L_2^2}, \frac{\epsilon_1^{2}}{2L_1}\right)\leq f(\x_j) - f(\x_{j+1})
\end{align*}
Case 2:  $\|\nabla f(\x_j)\|\leq\epsilon_1$ and $\v_j^{\top}\nabla^2 f(\x_j)\v_j\leq  -\epsilon_2/2$, we have
\begin{align*}
\frac{\epsilon_2^3}{12L_2^2}\leq f(\x_j) - f(\x_{j+1})
\end{align*}
Case 3:  $\|\nabla f(\x_j)\|> \epsilon_1$ and $\v_j^{\top}\nabla^2 f(\x_j)\v_j>  - \epsilon_2/2$, we have
\begin{align*}
\frac{\epsilon_1^2}{2L_1}\leq f(\x_j) - f(\x_{j+1})
\end{align*}
In any case, we have
\begin{align*}
\min\left(\frac{\epsilon_1^2}{2L_1}, \frac{\epsilon_2^3}{12L_2^2}\right)\leq f(\x_j) - f(\x_{j+1})
\end{align*}
Then with at most $j_* = 1 +  \max\left(\frac{12L_2^2}{\epsilon_2^3}, \frac{2L_1}{\epsilon_1^2}\right)\Delta$, the algorithm terminates. 
Note that $\epsilon_2=\epsilon_1^\alpha$, we know that $j_* = 1 +  \max\left(\frac{12L_2^2}{\epsilon_1^{3\alpha}}, \frac{2L_1}{\epsilon_1^2}\right)\Delta$.

Upon termination, we have with probability at least $1-j_*\delta'$, i.e. with probability at least $1-\delta$, 
\[
\lambda_{\min}(\nabla^2 f(\x_{j_*}))\geq -\epsilon_2/2 - \max(\epsilon_2, \|\nabla f(\x_{j_*})\|^\alpha)/2=-\epsilon_1^\alpha/2 - \max(\epsilon_1^\alpha, \|\nabla f(\x_{j_*})\|^\alpha)/2.
\]
Since $\|\nabla f(\x_{j_*})\|\leq\epsilon_1$, we have $$\max(\epsilon_1^\alpha, \|\nabla f(\x_{j_*})\|^\alpha)=\epsilon_1^\alpha,$$ and hence
$\lambda_{\min}(\nabla^2 f(\x_{j_*}))\geq-\epsilon_1^\alpha.$

The running time spent on the $j$-th iteration follows from Lemma~\ref{lemma:lanczos}.
\end{proof}

\section{Proof of Theorem \ref{thm:gnc-a3:iteration}}
Before diving into the proofs, we first present the procedure Almost-Convex-AGD and  introduce some propositions which are useful for our further analysis. 
\begin{algorithm}[H]
\caption{Almost-Convex-AGD$(f,\z_1,\epsilon,\gamma,L_1)$}
\begin{algorithmic}[1]
\FOR{$j=1,2,\ldots$}
\IF{$\|\nabla f(\z_j)\|\leq \epsilon$}
\RETURN $\z_j$
\ENDIF
\STATE Define $g_j(\z)=f(\z)+\gamma\|\z-\z_j\|^2$
\STATE set $\epsilon'=\epsilon\sqrt{\gamma/50(L_1+2\gamma)}$
\STATE $\z_{j+1}=\text{Accelerated-Gradient-Descent}(g_j,\z_j,\epsilon',L_1,\gamma)$
\ENDFOR
\end{algorithmic}
\end{algorithm}

\begin{algorithm}[H]
\caption{Accelerated-Gradient-Descent$(f,\y_1,\epsilon,L_1,\sigma_1)$}
\begin{algorithmic}[1]
\STATE Set $\kappa=L_1/\sigma_1$, $\z_1=\y_1$
\FOR{$j=1,2,\ldots$}
\IF{$\|\nabla f(\y_j)\|\leq\epsilon$}
\RETURN $\y_j$
\ENDIF
\STATE $\y_{j+1}=\z_j-\frac{1}{L_1}\nabla f(\z_j)$
\STATE $\z_{j+1}=(1+\frac{\sqrt{\kappa}-1}{\sqrt{\kappa}+1})\y_{j+1}-\frac{\sqrt{\kappa}-1}{\sqrt{\kappa}+1}\y_j$
\ENDFOR
\end{algorithmic}
\end{algorithm}

\begin{prop}[Lemma 3.1 of~\cite{DBLP:journals/corr/CarmonDHS16}]
\label{prop:almostconvexagd}
    Let $f:\R^d\rightarrow\R$ be $\gamma$-almost convex and $L_1$-smooth, where  $0<\gamma\leq L_1$. Then Almost-Convex-AGD$(f,\z_1,\epsilon,\gamma,L_1)$ returns a vector $\z$ such that $\|\nabla f(\z)\|\leq\epsilon$ and
    \begin{align}
    \label{ineqn:progressive}
        f(\z_1)-f(\z)\geq\min\left\{\gamma\|\z-\z_1\|^2,\frac{\epsilon}{\sqrt{10}\|\z-\z_1\|}\right\}
    \end{align}
    in time
    \begin{align}
    \label{ineqn:almostconvexagdtime}
    O\left(T_g\left(\sqrt{\frac{L_1}{\gamma}}+\frac{\sqrt{\gamma L_1}}{\epsilon^2}(f(\z_1)-f(\z))\right)\log\left(2+\frac{L_1^3\Delta}{\gamma^2\epsilon^2}\right)\right)
    \end{align}
\end{prop}
\begin{prop}[Lemma 4.1 of~\cite{DBLP:journals/corr/CarmonDHS16}]
\label{prop:almostconvexagd2}
Let $f$ be $L_1$-smooth and have $L_2$-Lipschitz continuous Hessian. Let $\x_0\in\R^d$ be such that $\nabla^2 f(\x_0)\succeq-\alpha I$ for some $\alpha\geq 0$, then $f(\x)+L_1\left[\|\x\|-\frac{\alpha}{L_2}\right]_+$
is $3\alpha$-almost convex and $5L_1$-smooth.
\end{prop}
The next result is a corollary of Theorem~\ref{lemma:GNC-A}, showing that by running $\xh_k = \text{NCG-A1}(\x_k, \epsilon_2^{3/2}, \epsilon_2, \delta')$ we obtain a solution $\xh_k$ around which $f(\x)$ is locally almost convex, i.e., $\nabla^2 f(\xh_k)\succeq -\epsilon_2 I$. 
\begin{corollary}
\label{cor:GNC-A}
The sub-routine $\xh_k = \text{NCG-A1}(\x_k, \epsilon_2^{3/2}, \epsilon_2, \delta')$ guarantees that 
\[
\lambda_{\min}(\nabla^2 f(\xh_k))\geq - \epsilon_2
\]
with at most $j_k$ iterations within NCG-A1, where
\begin{equation}\label{eqn:bound-GNC-A-cor}
j_k\leq 1 + \frac{\max(12L_2^2, 2L_1)}{\epsilon_2^3}(f(\x_k) - f(\xh_k))\leq 1 + \frac{\max(12L_2^2, 2L_1)}{\epsilon_2^3}\Delta,
\end{equation}
Furthermore, each iteration $j$ within NCG-A1 requires time at most 
\[
O\left(T_h\frac{\sqrt{L_1}}{\max(\epsilon_2, \|\nabla f(\x_j)\|)^{1/2}}\log\left(\frac{d}{\delta'}\right)\right)
\]
\end{corollary}

\begin{proof}[Proof of Theorem \ref{thm:gnc-a3:iteration}]
\begin{itemize}
\item First we try to bound the number of iterations in the Algorithm NCG-B1, which is actually the upper bound of the number of calls of both NCG-A1 and Almost-Convex-AGD.

	Define $\rho_{\alpha}(\x):=L_1\left[\|\x\|-\frac{\alpha}{L_2}\right]_+$.
	At iteration $k$ when $\|\nabla f(\widehat{\x}_k)\|\leq \epsilon_1$ is not met, which means $\|\nabla f(\widehat{\x}_k)\|>\epsilon_1$, we have
	\begin{align*}
		\epsilon_1<\|\nabla f(\widehat{\x}_k)\|\leq\left[\|\nabla f_{k-1}(\widehat{\x}_k)\|+\|\nabla\rho_{\epsilon_2}(\widehat{\x}_k-\widehat{\x}_{k-1})\|\right]\leq\frac{\epsilon_1}{2}+2L_1\left[\|\widehat{\x}_k-\widehat{\x}_{k-1}\|-\frac{\epsilon_2}{L_2}\right]_{+},
	\end{align*}
where the second inequality holds due to the triangle inequality, and the third inequality holds because of the guarantee provided by Almost-Convex-AGD at the previous stage. Therefore, we have
%	which implies that
	\begin{align}
		\label{eqn:1}
		\frac{\epsilon_1}{4L_1}\leq\left[\|\widehat{\x}_k-\widehat{\x}_{k-1}\|-\frac{\epsilon_2}{L_2}\right]_{+}=\|\widehat{\x}_k-\widehat{\x}_{k-1}\|-\frac{\epsilon_2}{L_2}.
	\end{align}
%	Since $\epsilon_2<\min(L_1^2,1)$, we have $0<\epsilon_2^{3/4}<\sqrt{\epsilon_2}<L_1$. 
	
%	By Lemma \ref{lemma:GNC-A} and union bound, we have the following results hold with probability at least $1-\delta$.
	
	According to the inequality (\ref{eqn:1}), we know that at iteration $1<k\leq K$, exactly one of the following three cases is true:
	\begin{enumerate}[label=(\Roman*)]
		\item  $\|\nabla f(\widehat{\x}_k)\|\leq\epsilon_1$ and the Algorithm NCG-B1 terminates
		\item  $\|\nabla f(\widehat{\x}_k)\|>\epsilon_1$ (which implies that $\|\widehat{\x}_k-\widehat{\x}_{k-1}\|\geq\frac{\epsilon_2}{L_2}$ according to (\ref{eqn:1})), and $\widehat{\x}_k\neq \x_k$
		\item  $\|\nabla f(\widehat{\x}_k)\|>\epsilon_1$ and $\widehat{\x}_k=\x_k$
	\end{enumerate}
	If (II) holds, note that the subroutine NCG-A1 needs at least 2 iterations, so according to Theorem \ref{lemma:GNC-A}, we have 
	\begin{align*}
		\max\left(\frac{12L_2^2}{\epsilon_2^3},\frac{2L_1}{\epsilon_2^3}\right)\left(f(\x_k)-f(\widehat{\x}_k)\right)\geq 1.
	\end{align*}
	Combining it with the progressive bound (\ref{ineqn:progressive}) in Proposition \ref{prop:almostconvexagd}, we have
	\begin{align*}
		f(\widehat{\x}_{k-1})-f(\widehat{\x}_k)\geq f(\x_k)-f(\widehat{\x}_k)\geq \min\left(\frac{\epsilon_2^3}{12L_2^2},\frac{\epsilon_2^3}{2L_1}\right).
	\end{align*}
	If (III) holds, then by Proposition \ref{prop:almostconvexagd2} and the second-order guarantee provided by Theorem \ref{lemma:GNC-A}, we can know that, with probability at least $1-\delta'$, $f_k$ is $3\epsilon_2$-almost convex and $5L_1$-smooth. Then applying Proposition \ref{prop:almostconvexagd} suffices to show that 
	\begin{align*}
		f(\widehat{\x}_{k-1})-f(\widehat{\x}_k)&\geq\min\left\{3\epsilon_2\|\widehat{\x}_{k-1}-\x_k\|^2,\frac{\epsilon_1}{2\sqrt{10}}\|\widehat{\x}_{k-1}-\x_k\|\right\}\\
		&\geq
		\min\left\{\frac{3\epsilon_2^3}{L_2^2},\frac{\epsilon_1\epsilon_2}{2\sqrt{10}L_2}\right\}.
	\end{align*}
	Combing two cases (II) and (III) together, we get the conclusion that whether in case (II) or case (III), with probability at least $1-\delta'$,
	
\begin{align*}
	f(\widehat{\x}_{k-1})-f(\widehat{\x}_{k})\geq\min\left(\frac{\epsilon_2^3}{12L_2^2},\frac{\epsilon_2^3}{2L_1},\frac{\epsilon_1\epsilon_2}{2\sqrt{10}L_2}\right).
\end{align*}

In order to get a contradiction that after $K$ iterations the algorithm has not terminated yet, and by the definition of $\delta'$ and union bound, it follows that, with probability at least $1-\delta$,
\begin{align*}
	\Delta \geq f(\widehat{\x}_1)-f(\widehat{\x}_K)=\sum_{k=1}^{K-1}(f(\widehat{\x}_k)-f(\widehat{\x}_{k+1}))\geq (K-1)\min\left(\frac{\epsilon_2^3}{12L_2^2},\frac{\epsilon_2^3}{2L_1},\frac{\epsilon_1\epsilon_2}{2\sqrt{10}L_2}\right).
\end{align*}
Plugging in $K=\lceil1+\Delta\left(\frac{\max(12L_2^2,2L_1)}{\epsilon_2^3}+\frac{2\sqrt{10}L_2}{\epsilon_1\epsilon_2}\right)\rceil$ suffices to get a contradiction. Therefore the algorithm terminates after at most $K$ outer iterations.
\item 
Define $\tau=1+1/\epsilon+1/\delta+d+L_1+L_2+\Delta$. We try to bound the number of NCG steps. Denote $j_{k}$ by the total number of times the Algorithm NCG-A1 is executed during the iteration $k$ of the method NCG-B1, and define $k^*$ as the total number of outer iterations of the Algorithm NCG-B1. By telescoping bound (\ref{eqn:bound-GNC-A-cor}) and the progressive bound (\ref{ineqn:progressive}) of Proposition \ref{prop:almostconvexagd} in Almost-Convex-AGD, which guarantees the Almost-Convex-AGD decreases the function values, we have
\begin{align*}
	\sum_{k=1}^{k^*}(j_{k}-1)&\leq\sum_{k=1}^{k^*}\max\left(\frac{12L_2^2}{\epsilon_2^3},\frac{2L_1}{\epsilon_2^3}\right)(f(\x_k)-f(\widehat{\x}_k))\\
	%&=\sum_{k=1}^{k^*}\max\left(12L_2^2,2L_1\right)(f(\x_k)-f(\widehat{\x}_k))\epsilon^{-3/2}\\
	&\leq \sum_{k=1}^{k^*}\max\left(\frac{12L_2^2}{\epsilon_2^3},\frac{2L_1}{\epsilon_2^3}\right)(f(\widehat{\x}_{k-1})-f(\widehat{\x}_k))\\
	&\leq \max\left(\frac{12L_2^2}{\epsilon_2^3},\frac{2L_1}{\epsilon_2^3}\right)\Delta.	
\end{align*}
According to the previous result, with probability at least $1-\delta$, we can have a upper bound of $k^*$, which is
\begin{align}
	\label{k:upperbound}
	k^*\leq 2+\Delta\left(\frac{12L_2^2}{\epsilon_2^3}+\frac{2L_1}{\epsilon_2^3}+\frac{2\sqrt{10}L_2}{\epsilon_1\epsilon_2}\right).
\end{align}
Hence, we have with probability at least $1-\delta$,
\begin{align}
	\label{bound:GNC}
	\sum_{k=1}^{k^*}j_k=k^*+\sum_{k=1}^{k^*}(j_k-1)\leq 2+\Delta\left(\frac{24L_2^2}{\epsilon_2^3}+\frac{4L_1}{\epsilon_2^3}+\frac{2\sqrt{10}L_2}{\epsilon_1\epsilon_2}\right). 
	\end{align}
According to Corollary~\ref{cor:GNC-A}, we have the cost of each iteration $t$ within NCG-A1  is
\begin{align*}
	O\left(T_h\frac{\sqrt{L_1}}{\max\left(\epsilon_2,\|\nabla f(\x_t)\|\right)^{1/2}}\log\left(\frac{d}{\delta'}\right)\right).
\end{align*}
Note that the failure probability satisfies
\begin{align*}
	\frac{1}{\delta'}\leq\frac{2+\Delta\left(\frac{12L_2^2}{\epsilon_2^3}+\frac{2L_1}{\epsilon_2^3}+\frac{2\sqrt{10}L_2}{\epsilon_1\epsilon_2}\right)}{\delta},
\end{align*}
so $\log\frac{d}{\delta'}=O(\log\tau)$.
Then we employ (\ref{bound:GNC}) to bound the worst-case total costs of NCG-A1, which is
%\begin{align*}
%	O\left(T_g\Delta(L_2^2+4L_1+L_2)\sqrt{L_1}(\epsilon^{-7/4}+\epsilon^{-3/2}\min_{j}\|\nabla f(\x_j)\|^{1/2})\log\tau\right)=\widetilde{O}(\epsilon^{-7/4}).
%\end{align*}
\begin{align}
\label{GNC-A3:c1}
	O\left(T_h\frac{\sqrt{L_1}}{\epsilon_2}\left[2+\Delta\left(\frac{24L_2^2}{\epsilon_2^3}+\frac{4L_1}{\epsilon_2^3}+\frac{2\sqrt{10}L_2}{\epsilon_1\epsilon_2}\right)\right]\log \tau\right)
\end{align}
Now we analyze the total cost of calling Almost-Convex-AGD. Employing the bound (\ref{ineqn:almostconvexagdtime}) in Proposition \ref{ineqn:almostconvexagdtime}, the cost of calling Almost-Convex-AGD in iteration $k$ with almost convexity parameter $3\epsilon_2$ is bounded by 
\begin{align*}
	O\left(T_g\left(\sqrt{\frac{L_1}{3\epsilon_2}}+\frac{4\sqrt{3\epsilon_2L_1}}{\epsilon_1^2}[f_k(\x_k)-f_k(\x_{k+1})]\right)\log\tau\right).
	\end{align*}
Note that $f_k(\x_k)-f_k(\x_{k+1})\leq f(\x_k)-f(\x_{k+1})$, so we have
\begin{align*}
	\sum_{k=1}^{k^*}\left[f_k(\x_k)-f_k(\x_{k+1})\right]\leq\sum_{k=1}^{k^*}[f(\x_k)-f(\x_{k+1})]\leq\Delta.
\end{align*}
According to (\ref{k:upperbound}), we can get that the total time complexity of Almost-Convex-AGD is %$\widetilde{O}(\epsilon^{-7/4})$.
\begin{align}
\label{GNC-A3:c2}
	O\left(T_g\left[\sqrt{\frac{L_1}{3\epsilon_2}}\left(2+\Delta\left(\frac{24L_2^2}{\epsilon_2^3}+\frac{4L_1}{\epsilon_2^3}+\frac{2\sqrt{10}L_2}{\epsilon_1\epsilon_2}\right)\right)+\frac{4\sqrt{3\epsilon_2L_1}}{\epsilon_1^2}\Delta\right]\log\tau\right).
\end{align}
According to (\ref{GNC-A3:c1}) and (\ref{GNC-A3:c2}), and note that $T_g\leq T_h$, we get that the worst case complexity bound is
$$\widetilde O\left(\left(\frac{1}{\epsilon_1\epsilon_2^{3/2}} +\frac{1}{\epsilon_2^{7/2}}\right) T_h+\frac{\epsilon_2^{1/2}}{\epsilon_1^2}T_g\right),$$
where $\widetilde{O}(\cdot)$ hides a $\log\tau$ factor.
\end{itemize}
\end{proof}

\section{Proof of Theorem \ref{thm:gnc-a3':iteration}}
The proof is similar to that of Theorem \ref{thm:gnc-a3:iteration} by observing the following result, which is a corollary of Theorem~\ref{thm:gnc-a2:iteration}. 
\begin{corollary}
\label{cor:GNC-A2}
The sub-routine $\xh_k = \text{NCG-A2}(\x_k, \epsilon_1^{3\alpha/2}, \frac{2}{3}, \delta')$ guarantees that 
\[
\lambda_{\min}(\nabla^2 f(\xh_k))\geq - \epsilon_2
\]
with at most $j_k$ iterations within NCG-A2, where
\begin{equation*}
j_k\leq 1 + \frac{\max(12L_2^2, 2L_1)}{\epsilon_1^{3\alpha}}(f(\x_k) - f(\xh_k))\leq 1 + \frac{\max(12L_2^2, 2L_1)}{\epsilon_1^{3\alpha}}\Delta,
\end{equation*}
Furthermore, each iteration $j$ within NCG-A2 requires time at most 
\[
O\left(T_h\frac{\sqrt{L_1}}{\max(\epsilon_2, \|\nabla f(\x_j)\|^{2/3})^{1/2}}\log\left(\frac{d}{\delta'}\right)\right)
\]
\end{corollary}

\section{Proof of Lemma \ref{lem:NCG-iH}}
\begin{proof}
Define $\eta_j=\frac{\epsilon_2}{L_2}\text{sign}(\v_j^\top\nabla f(\x_j))$, and then by $L_2$-Lipschitz continuity of Hessian and the fact that $\eta_j\v_j^\top\nabla f(\x_j)\geq 0$, we know that
\begin{align*}
	f(\x_{j+1}^1)-f(\x_j)\leq\frac{\eta_j^2}{2}\v_j^\top \left(\nabla^2 f(\x_j)-H(\x_j)\right)\v_j+\frac{\eta_j^2}{2}\v_j^\top H(\x_j)\v_j+\frac{L_2}{6}|\eta_j|^3 .
	\end{align*}
	where $\x_{j+1}^1$ is an update of $\x_j$ following $\v_j$ in NCG-iH. 
%By taking $\epsilon=\frac{\epsilon_2}{12}$ in the proposition \ref{proposition1} and noting that $|\mathcal{S}|\geq \frac{2304K_{\text{max}}^2}{\epsilon_2^2}\log\frac{4d}{\delta'}$, 
By the condition  $\v_j^\top \left(\nabla^2 f(\x_j)-H(\x_j)\right)\v_j\leq \epsilon_3\leq \epsilon_2/12$, we can derive that% with probability at least $1-\delta'/2$,
\begin{align}
	f(\x^1_{j+1})-f(\x_j)\leq \frac{\epsilon_2^3}{24L_2^2}+\frac{\epsilon_2^2\v_j^\top H(\x_j)\v_j}{2L_2^2}+\frac{\epsilon_2^3}{6L_2^2}=-\underbrace{\left(\frac{-\epsilon_2^2\v_j^\top H(\x_j)\v_j}{2L_2^2} - \frac{5\epsilon_2^3}{24L_2^2}\right)}\limits_{\Delta_1}.
\end{align}
Similarly, let $\x^2_{j+1}$ denote an update of $\x_j$ following $\nabla f(\x_j)$ in NCG-iH, we have
\begin{align}
	f(\x^2_{j+1})-f(\x_j)\leq -\underbrace{\frac{\|\nabla f(\x_j)\|^2}{2L_1}}\limits_{\Delta_2}.
\end{align}
According to the update of NCG-iH, if $\Delta_1>\Delta_2$, we have $\x_{j+1} = \x^1_{j+1}$ and then $f(\x_{j}) - f(\x_{j+1})\geq \Delta_1 = \max(\Delta_1, \Delta_2)$. If $\Delta_2\geq \Delta_1$, we have $\x_{j+1} = \x^2_{j+1}$ and then $f(\x_{j}) - f(\x_{j+1})\geq \Delta_2 = \max(\Delta_1, \Delta_2)$, which finishes the proof of the first part.   When $\v_j^\top H(\x_j)\v_j\leq -\epsilon_2/2$, we have $\Delta_1\geq \frac{\epsilon_2^3}{24L_2^2}$ and
\begin{align*}
f(\x_j) - f(\x_{j+1})\geq \max\left(\frac{\|\nabla f(\x_j)\|^2}{2L_1}, \frac{\epsilon_2^3}{24L_2^2}\right)
\end{align*}

\end{proof}
\section{Proof of Lemma~\ref{lem:approxPCA}}
We first introduce a proposition, which is the Theorem 2.5 in~\cite{DBLP:conf/stoc/AgarwalZBHM17}.
\begin{prop}
\label{agarwal:thm2.5}
Let $M\in\R^{d\times d}$ be a symmetric matrix with eigenvalues $1\geq\lambda_1\ldots\geq\lambda_d\geq 0$. Then with probability at least $1-p$, the Algorithm AppxPCA produces a unit vector $\v$ such that $\v^\top M\v\geq(1-\delta_+)(1-\epsilon)\lambda_{\text{max}}(M)$. The total running time is $\widetilde O\left(T_h^1\max\{m,\frac{m^{3/4}}{\sqrt{\epsilon}}\}\log^2\left(\frac{1}{\epsilon^2\delta_+}\right)\right)$.
\end{prop}
\begin{proof}[Proof of Lemma~\ref{lem:approxPCA}]
Define $M=I-\frac{H}{L_1}$, then $M$ satisfies the condition in the Proposition \ref{agarwal:thm2.5}. Then we know that with probability at least $1- p$, the Algorithm AppxPCA produces a vector $\v$ satisfying
\begin{align*}
    \v^\top\left(I-\frac{H}{L_1}\right)\v\geq(1-\delta_+)(1-\epsilon)\left(1-\frac{\lambda_{\text{min}}(H)}{L_1}\right),
\end{align*}
which implies that
\begin{align*}
L_1-\v^\top H\v\geq (1-\delta_+ -\epsilon + \delta_+\epsilon)(L_1-\lambda_{\text{min}}(H))\geq (1-\delta_+ -\epsilon)(L_1-\lambda_{\text{min}}(H)).
\end{align*}
By simple algebra, we have
\begin{align*}
    \lambda_{\text{min}}(H)&\geq\v^\top H\v-(\delta_++\epsilon)(L_1-\lambda_{\text{min}}(H))\\
    &\geq\v^\top H\v-2L_1(\delta_++\epsilon).
\end{align*}
By setting $\epsilon=\delta_+=\frac{\varepsilon}{4L_1}$, we can finish the proof.  
\end{proof}

\section{Proof of Theorem \ref{thm:GNC-A-stoc}}
\begin{proof}
Define $\eta_j=\frac{\epsilon_2}{L_2}\text{sign}(\v_j^\top\g(\x_j))$. Next, we analyze the objective decrease for $j$-th  NCG-S step conditioned on the event $\mathcal A=\{\|H(\x_j) - \nabla^2 f(\x_j)\|_2\leq \epsilon_3 \cap \|g(\x_j) - \nabla f(\x_j)\|\leq \epsilon_4\}$. 

By $L_2$-Lipschitz continuity of Hessian, 
%and the fact that $\eta_j\v_j^\top\nabla f(\x_j)\geq 0$, 
we know that
\begin{align*}
	f(\x_{j+1}^1)-f(\x_j)\leq -\eta_j\nabla f(\x_j)^{\top}\v_j + \frac{\eta_j^2}{2}\v_j^\top \left(\nabla^2 f(\x_j)-H(\x_j)\right)\v_j+\frac{\eta_j^2}{2}\v_j^\top H(\x_j)\v_j+\frac{L_2}{6}|\eta_j|^3 .
	\end{align*}
	where $\x_{j+1}^1$ is an update of $\x_j$ following $\v_j$ in NCG-S. 
%By taking $\epsilon=\frac{\epsilon_2}{12}$ in the proposition \ref{proposition1} and noting that $|\mathcal{S}|\geq \frac{2304K_{\text{max}}^2}{\epsilon_2^2}\log\frac{4d}{\delta'}$, 
By the assumption $\epsilon_4\leq \epsilon_2^2/24L_2$ we have
\begin{align}
    & -\eta_j\nabla f(\x_j)^{\top}\v_j =  -\eta_j \g(\x_j)^{\top}\v_j + \eta_j(\g(\x_j) - \nabla f(\x_j))^{\top}\v_j\leq |\eta_j\epsilon_4| \leq \frac{\epsilon_2^3}{24L_2^2}\\
    & \v_j^\top \left(\nabla^2 f(\x_j)-H(\x_j)\right)\v_j\leq \epsilon_3\leq \epsilon_2/24
\end{align}
Then it follows that% with probability at least $1-\delta'/2$,
\begin{align}
	f(\x^1_{j+1})-f(\x_j)\leq \frac{\epsilon_2^3}{24L_2^2}+\frac{\epsilon_2^3}{48L_2^2}+\frac{\epsilon_2^2\v_j^\top H(\x_j)\v_j}{2L_2^2}+\frac{\epsilon_2^3}{6L_2^2}=-\underbrace{\left(\frac{-\epsilon_2^2\v_j^\top H(\x_j)\v_j}{2L_2^2} - \frac{11\epsilon_2^3}{48L_2^2}\right)}\limits_{\Delta_1}.
\end{align}
Similarly, let $\x^2_{j+1}$ denote an update of $\x_j$ following $\g(\x_j)$ in NCG-S,  we have
\begin{align*}
	f(\x^2_{j+1})-f(\x_j)&\leq (\x^2_{j+1} - \x_j)^{\top}\nabla f(\x_j) + \frac{L_1}{2}\|\x^2_{j+1} - \x_j\|^2\\
	&= -\frac{1}{L_1}\g(\x_j)^{\top}\nabla f(\x_j) + \frac{\|\g(\x_j)\|^2}{2L_1}\\
	&=-\frac{1}{L_1}\g(\x_j)^{\top}\g(\x_j) + \frac{1}{L_1}\g(\x_j)^{\top}(\g(\x_j) - \nabla f(\x_j)) +  \frac{\|\g(\x_j)\|^2}{2L_1}\\
	&\leq - \frac{1}{2L_1}\|\g(\x_j)\|^2  + \frac{1}{4L_1}\|\g(\x_j)\|^2 + \frac{1}{L_1}\|\g(\x_j) - \nabla f(\x_j)\|^2\\
	&= - \frac{1}{4L_1}\|\g(\x_j)\|^2  + \frac{1}{L_1}\epsilon_4^2\leq - \underbrace{\frac{1}{4L_1}\|\g(\x_j)\|^2  + \frac{\epsilon_1^2}{8L_1}}\limits_{-\Delta_2}\\
\end{align*}
where we use $\epsilon_4\leq \frac{1}{2\sqrt{2}}\epsilon_1$. 
According to the update of NCG-S, if $\Delta_1>\Delta_2$, we have $\x_{j+1} = \x^1_{j+1}$ and then $f(\x_{j}) - f(\x_{j+1})\geq \Delta_1 = \max(\Delta_1, \Delta_2)$. If $\Delta_2\geq \Delta_1$, we have $\x_{j+1} = \x^2_{j+1}$ and then $f(\x_{j}) - f(\x_{j+1})\geq \Delta_2 = \max(\Delta_1, \Delta_2)$. Therefore, with probability $1- \delta'$ we have, 
\begin{align*}
f(\x_j) - f(\x_{j+1})\geq \max\left(\frac{1}{4L_1}\|\g(\x_j)\|^2   -  \frac{\epsilon_1^2}{8L_1}, \frac{-\epsilon_2^2\v_j^\top H(\x_j)\v_j}{2L_2^2} - \frac{11\epsilon_2^3}{48L_2^2}\right)
\end{align*}
If $\v_j^\top H(\x_j)\v_j\leq -\epsilon_2/2$, we have $\Delta_1\geq \frac{\epsilon_2^3}{48L_2^2}$ and
\begin{align*}
f(\x_j) - f(\x_{j+1})\geq  \frac{\epsilon_2^3}{48L_2^2}
\end{align*}
If $\|\g(\x_j)\|> \epsilon_1$, we have $\Delta_2\geq \frac{\epsilon_1^2}{8L_1}$ and 
\begin{align*}
f(\x_j) - f(\x_{j+1})\geq  \frac{\epsilon_1^2}{8L_1}
\end{align*}
Therefore, before the algorithm terminates, i.e., for all iterations $j\leq j_* -1$, we have either $\v_j^\top H(\x_j)\v_j\leq -\epsilon_2/2$ or $\|\g(\x_j)\|> \epsilon_1$. In either case, the following holds with probability $1-\delta'$
\begin{align*}
f(\x_j) - f(\x_{j+1})\geq  \min\left(\frac{\epsilon_1^2}{8L_1}, \frac{\epsilon_2^3}{48L_2^2}\right),
\end{align*}
from which we can derive the upper bound of $j_*$ as in the theorem. Next, we show that upon termination, we achieve an $(2\epsilon_1, 2\epsilon_2)$-second order stationary point with high probability. In particular, with probability $1-\delta'$ we have
\[
\|\nabla f(\x_{j_*})\|\leq \|\nabla f(\x_{j_*}) - \g(\x_{j_*})\| + \|\g(\x_{j_*})\| \leq \epsilon_4  + \epsilon_1 \leq 2\epsilon_1. 
\]
and  with probability $1-\delta'$
\begin{align*}
\lambda_{\min}(H(\x_{j_*}))\geq \v_{j_*}^{\top}H(\x_{j_*})\v_{j_*} - \max(\epsilon_2, \|\g(\x_{j_*})\|^\alpha)/2 \geq -\epsilon_2 
\end{align*}
In addition, with probability $1-\delta'$, we have
\[
\lambda_{\min}(\nabla^2 f(\x_{j_*}))\geq \lambda_{\min}(H(\x_{j_*})) - \epsilon_3 
\]
As a result, by using union bound, we have with probability $1-3j_*\delta'=1-3\delta$, we have
\begin{align*}
  \|\nabla f(\x_{j_*})\|\leq2\epsilon_1, \quad   \lambda_{\min}(\nabla^2 f(\x_{j_*}))\geq -2\epsilon_2
\end{align*}
Finally, the time complexity of each iteration follows Lemma~\ref{lem:approxPCA}. 
\end{proof}

%\section{Proof of Lemma~\ref{lem:gc}}
%The lemma is a corollary of the following concentration result. 
%\begin{prop} \citep{Smale:learning}. \label{lem:con} Let $\H$ be a Hilbert space equipped with a norm $\|\cdot\|$ and let $\xi$ be a random variable with values in $\H$. Assume $\|\xi\|\leq M < \infty$ almost surely. Denote $\sigma^2(\xi)=\mathbb{E}\left[\|\xi\|^2\right]$. Let  $\{\xi_i\}_{i=1}^m$ be $m$ ($m < \infty$) independent drawers of $\xi$. For any $0 < \delta < 1$, with confidence $1-\delta$,
%\[
%\left\| \frac{1}{m} \sum_{i=1}^m \left[\xi_i -\E[\xi_i]\right] \right\| \leq \frac{2 M \log(2/\delta)}{m} + \sqrt{\frac{2 \sigma^2(\xi) \log(2/\delta)}{m}}.
%\]
%\end{prop}
%Note that $\|\nabla f_i(\x)\|\leq G$ for any $\x\in\R^d$, then by the Proposition \ref{lem:con}, we know that with probability at least $1-\delta$, we have
%\begin{align*}
%    \|\g(\x)-\nabla f(\x)\|\leq\frac{2G\log(2/\delta)}{|\mathcal{S}_1|}+\sqrt{\frac{2G^2\log(2/\delta)}{|\mathcal{S}_1|}}.
%\end{align*}
%When $|\mathcal{S}_1|\geq 2\log(2/\delta)$, we have
%\begin{align*}
%    \|\g(\x)-\nabla f(\x)\|\leq 2G\sqrt{\frac{2\log(2/\delta)}{|\mathcal{S}_1|}}.
%\end{align*}
%Furthermore, when  $|\mathcal{S}_1|\geq(2+\frac{8G^2}{\epsilon_4^2})\log(2/\delta)\geq\max\left(2\log(2/\delta),\frac{8G^2}{\epsilon_4^2}\log(2/\delta)\right)$, we have
%\begin{align*}
%    \|\g(\x)-\nabla f(\x)\|\leq2G\sqrt{\frac{2\log(2/\delta)}{|\mathcal{S}_1|}}\leq\epsilon_4.
%\end{align*}

\end{document}